%% file: main-DNNmatching.tex
\documentclass[a4paper]{article}

\usepackage{graphicx}
\usepackage{amsmath,amssymb,amsfonts,amsthm,bm,mathtools}
\mathtoolsset{showonlyrefs=true,showmanualtags=true}
\usepackage{enumitem}
\usepackage{booktabs}
\usepackage[T1]{fontenc}
\usepackage{lmodern}
\usepackage[numbers,sort&compress]{natbib}
\usepackage{url}
\usepackage[unicode]{hyperref}
\hypersetup{hidelinks}

\newtheorem{theorem}{Theorem}[section]
\newtheorem{lemma}[theorem]{Lemma}
\newtheorem{proposition}[theorem]{Proposition}
\newtheorem{corollary}[theorem]{Corollary}
\theoremstyle{definition}
\newtheorem{definition}[theorem]{Definition}
\theoremstyle{remark}
\newtheorem{remark}[theorem]{Remark}

\newcommand{\R}{\mathbb{R}}
\newcommand{\supp}{\operatorname{supp}}
\newcommand{\relu}{\operatorname{ReLU}}
\newcommand{\ideal}{\mathcal{I}}
\newcommand{\qmod}{\mathcal{M}}
\newcommand{\Uset}{\mathcal{J}_{\mathrm{u}}}

\numberwithin{equation}{section}

\begin{document}

\begin{center}
\large{\bf Order-$2$ Tightness of Block-Sparse SOS Relaxations\\
for One-Layer ReLU Network Verification\\
with a Matching Input-Sharing Graph}
\end{center}\vspace{5mm}

\begingroup
\renewcommand{\thefootnote}{\fnsymbol{footnote}}
\begin{center}
\normalsize
Godai Azuma
\quad
Sunyoung Kim
\quad
Makoto Yamashita\footnotemark[1]
\end{center}
\footnotetext[1]{Corresponding author.
The research of Godai Azuma was supported by JSPS KAKENHI Grant Number
JP24K20738. The research of Makoto Yamashita was partially supported by
JSPS KAKENHI Grant Number 24K14836 and 26K02868.
}
\endgroup
\setcounter{footnote}{0}

\vspace{2mm}

{\footnotesize
\noindent\begin{minipage}{14cm}
{\bf Abstract:}
Azuma, Kim, and Yamashita formulated the verification problem for one-layer ReLU networks as a quadratically constrained quadratic program and
 established tight semidefinite relaxations for the edgeless case and for one-unit settings.
In this work, we represent the sharing pattern of undecided ReLUs over a box input set through an input-sharing graph
and focus on the case where this graph is a matching. We then derive an explicit, checkable sufficient condition for
the tightness of the order-$2$ block-sparse SOS relaxation associated with the connected-component decomposition of this graph.
 Under the matching assumption,
 the global problem decomposes into isolated-vertex blocks and single-edge blocks.
 The key difficulty, which is absent from the edgeless case, is establishing tightness for a two-unit edge block.
 For regular rank-one edges, we show that the convex hull of each two-unit local set can be described exactly
 by two reduced one-unit hulls coupled through a common shared scalar. Combining the one-unit tightness
 result of Azuma et al. with Farkas' lemma and affine elimination, we obtain a local order-$2$ certificate for each edge block.
 Isolated-vertex blocks reduce to one-unit problems over box input sets and are therefore handled at the same order.
 We prove that, when the input-sharing graph is a matching and every edge satisfies the regular rank-one condition,
 the order-$2$ block-sparse SOS relaxation is tight. This extends the tight sparse relaxation result for the edgeless
 case to the first sparse setting with a nontrivial two-unit interaction.
\end{minipage}
\\[5mm]

\noindent{\bf Keywords:} ReLU network verification; sums-of-squares relaxation;
semidefinite programming; sparse polynomial optimization; input-sharing graph.\\
\noindent{\bf Mathematics Subject Classification:} 90C22; 90C26; 90C90; 68T07.

\hbox to14cm{\hrulefill}\par
}
\vspace{5mm}

\input{section-01}
\input{section-02}
\input{section-03}
\input{section-04}
\input{section-05}

\input{section-06}

\section*{Acknowledgements}
We used AI-based language tools to improve the wording, readability, and presentation of some parts of the manuscript.
We also used such tools at an early stage of writing to help generate possible examples.
All mathematical statements, proofs, formulations, and numerical results were checked independently by
the authors, without relying on AI-generated claims. The authors take full responsibility for the accuracy,
validity, and integrity of the contents of this paper.

\bibliographystyle{abbrvnat}
\bibliography{references-DNNmatching}

\appendix
\input{section-0a}

\clearpage

\bigskip

\noindent{\sc Godai Azuma}\\
Department of Mathematical and Computing Science,
Institute of Science Tokyo, 2-12-1-W8-29 Oh-Okayama, Meguro-ku,
Tokyo 152-8550, Japan\\
Department of Industrial and Systems Engineering,
Aoyama Gakuin University, 5-10-1 Fuchinobe, Chuo-ku, Sagamihara-shi,
Kanagawa 252-5258, Japan\\
E-mail address: {\tt azuma@comp.isct.ac.jp}

\bigskip

\noindent{\sc Sunyoung Kim}\\
Department of Mathematics, Ewha W. University, 52 Ewhayeodae-gil,
Seodaemun-gu, Seoul 03760, Korea\\
E-mail address: {\tt skim@ewha.ac.kr}

\bigskip

\noindent{\sc Makoto Yamashita}\\
Department of Mathematical and Computing Science,
Institute of Science Tokyo, 2-12-1-W8-29 Oh-Okayama, Meguro-ku,
Tokyo 152-8550, Japan\\
E-mail address: {\tt Makoto.Yamashita@comp.isct.ac.jp}

\end{document}

%% file: section-01.tex
\section{Introduction}\label{sec:intro}

Using neural networks in safety-critical applications requires rigorous bounds on their outputs,
or on associated specification functions, over given input sets.
Such bounds are essential for certifying, for instance, that admissible input perturbations
do not lead to misclassification, or that a system trajectory remains within a designated safe region.
Numerous verification methods have been proposed, including complete approaches based on
satisfiability modulo theories (SMT) and Boolean satisfiability (SAT), abstract interpretation,
branch-and-bound methods, mixed-integer optimization, semidefinite relaxations, and hybrid methods
\citep{KatzBarrettDillJulianKochenderfer2017,Ehlers2017,GehrMirmanDrachslerCohenTsankovChaudhuriVechev2018,BunelTurkaslanTorrKohliMudigonda2018,TjengXiaoTedrake2019,BunelLuTurkaslanTorrKohliKumar2020,FazlyabMorariPappas2022}.
For surveys, see \citet{LiuArnonLazarusStrongBarrettKochenderfer2021,MengBaiTeoHouXiaoLinDong2022};
for a tutorial perspective in control, see \citet{Everett2021}. A basic theoretical question is how accurately
a tractable relaxation approximates the original verification problem and which structural conditions
ensure tightness. 

In this work, we consider the one-layer ReLU network
\begin{align*}
 v &= Wu+b,\\
 x &= \relu(v),\\
 y &= Ax+a
\end{align*}
where $u\in\R^{n_0}$ is the input, $v\in\R^{n_1}$ is the
pre-activation vector, $x\in\R^{n_1}$ is the hidden-layer output,
$y\in\R^{n_2}$ is the final output,
and $x = \relu(v)$ denotes
$x_j=\max\{v_j,0\}$ for $j=1,\dots,n_1$.
We consider the box input set
\[
\mathcal U:=\{u\in\R^{n_0}\mid \underline u\le u\le \bar u\}.
\]
We study the verification problem of maximizing a linear specification $c^\top y$ over this set:
\begin{equation}
\gamma^*:=\max\{c^\top y\mid u\in\mathcal U,\ v=Wu+b,\ x=\relu(v),\ y=Ax+a\}.
\label{eq:intro-verification-problem}
\end{equation}
This form is a basic model that can represent, with auxiliary variables if needed, many specifications such as componentwise output bounds and margin-type inequalities between outputs.

As each ReLU constraint can be represented by linear inequalities and one quadratic equality, problem \eqref{eq:intro-verification-problem} is a quadratically constrained quadratic program (QCQP) and hence a polynomial optimization problem. One can then apply semidefinite relaxations based on sums of squares (SOS) and moment techniques \citep{Lasserre2001,Putinar1993,Laurent2009}. Moreover, sparse SOS relaxations that exploit correlative or term sparsity \citep{WakiKimKojimaMuramatsu2006,WangMagronLasserre2021,WangMagronLasserreMai2022} are relevant to structured polynomial optimization. In a related direction, \citet{MarumoKimYamashita2024} study a T-SDP relaxation for constrained polynomial optimization and show that it can be transformed into a standard SDP relaxation with block-diagonal structure. In the setting of neural network verification, \citet{NewtonPapachristodoulou2023} apply sparse polynomial optimization and the Positivstellensatz to obtain tighter bounds for ReLU, sigmoid, and tanh networks.

Our goal is to identify graph-based conditions under which a sparse relaxation is tight for one-layer ReLU verification.
The starting point for our analysis is the work of Azuma et al.~\citep{AKY2026},
which established tight semidefinite relaxations for one-layer ReLU networks in three settings:
the one-unit case, the ellipsoidal-input case, and the box-input case with a diagonal weight matrix.
 When the input set is a box and the weight matrix is diagonal,
 each undecided unit depends on input variables that are not shared with any other undecided unit.
 The problem therefore decomposes unit by unit. The diagonal-weight setting is therefore a special case of an edgeless input-sharing graph.

We study the matching case, in which each undecided ReLU shares input variables with at most one other undecided ReLU. Under this assumption, every connected component is either an isolated vertex or a single edge. Existing one-unit results handle isolated vertices, whereas a single edge requires a new tightness argument for two units coupled through a shared scalar.

For a regular rank-one edge, we describe the convex hull of the two-unit block and derive an SOS certificate at relaxation order $2$. Together with the reduced-objective condition introduced later, these local certificates prove tightness of the block-sparse relaxation for matching input-sharing graphs. Isolated vertices are handled by the one-unit result of Azuma et al.~\cite{AKY2026}.

The significance of this result lies in replacing a dense global certificate with local
 certificates supported on the graph structure of the problem.
 Tightness is established componentwise, through one-unit blocks and two-unit edge blocks.
In this sense, the paper contributes not only to ReLU network
verification, but also to graph-structured nonconvex polynomial optimization more broadly.

The main contributions of the paper are as follows.
\begin{enumerate}[label=(\roman*),leftmargin=2.4em]
\item We introduce the input-sharing graph to describe coupling among undecided ReLUs and relate the diagonal-weight case of Azuma et al.~\cite{AKY2026} to the edgeless setting.
\item Under the matching assumption, we decompose the reduced verification problem into isolated-vertex and single-edge blocks and formulate the corresponding block-sparse SOS relaxation.
\item For a regular rank-one edge, we obtain an exact convex-hull description and derive a local certificate at relaxation order $2$.
\item By assembling the one-unit and two-unit certificates, we prove global tightness under the matching and reduced-objective conditions.
\end{enumerate}

The paper is organized as follows. Section~\ref{sec:prelim} reviews the semialgebraic formulation of one-layer ReLU verification, the basic machinery of SOS relaxations, and the one-unit results of \citet{AKY2026}. Section~\ref{sec:graph} introduces determined units, undecided ReLUs, the input-sharing graph, matching, and regular rank-one edges. Section~\ref{sec:reduction} shows that under the matching condition the problem decomposes into isolated-vertex components and single-edge components, and defines the order-$2$ block-sparse SOS relaxation used in this paper. Section~\ref{sec:edge} proves local tightness for a regular rank-one edge and then states the main theorem for the matching case. Section~\ref{sec:conclusion} concludes the paper. Appendix~\ref{app:worked-example} gives a small example that illustrates the input-sharing graph, the component decomposition, the regular rank-one condition, and the resulting tightness statement.

%% file: section-02.tex
\section{Problem Formulation and Sparse Polynomial Relaxations via SOS}\label{sec:prelim}

\subsection{Semialgebraic formulation of the verification problem and reduction to polynomial optimization}

We consider the one-layer ReLU network
\begin{align}
 v &= Wu+b, \label{eq:dnn-v}\\
 x &= \relu(v), \label{eq:dnn-x}\\
 y &= Ax+a \label{eq:dnn-y},
\end{align}
where $u\in\R^{n_0}$ is the input, $v,x\in\R^{n_1}$ are the pre-activation variables and ReLU outputs, respectively, and $y\in\R^{n_2}$ is the output. The data are $W\in\R^{n_1\times n_0}$, $b\in\R^{n_1}$, $A\in\R^{n_2\times n_1}$, and $a\in\R^{n_2}$. We assume that the input set is the box
\begin{equation}
\mathcal U:=\{u\in\R^{n_0}\mid \underline u\le u\le \bar u\}.
\label{eq:input-box}
\end{equation}
Our basic problem is the upper-bound evaluation problem for a linear specification vector $c\in\R^{n_2}$,
\begin{equation}
\gamma^*:=\max\{c^\top y \mid u\in\mathcal U,\ v=Wu+b,\ x=\relu(v),\ y=Ax+a\}.
\label{eq:verification-problem}
\end{equation}
This form represents not only the maximization of an output component but also, for example, componentwise output bounds and margin-type constraints defined by differences between competing classes.

For each unit $j=1,\dots,n_1$, the ReLU constraint $x_j=\max\{v_j,0\}$ is equivalent to
\begin{equation}
 x_j\ge 0,
 \qquad
 x_j-v_j\ge 0,
 \qquad
 x_j(x_j-v_j)=0.
\label{eq:relu-semi-algebraic}
\end{equation}
Hence problem \eqref{eq:verification-problem} can be written as a semialgebraic optimization problem consisting of affine relations, linear inequalities, and quadratic equalities. Let us collect all variables as
\[
 z:=(u,v,x,y).
\]
We write the inequality constraints as $g_1(z),\dots,g_m(z)\ge 0$ and the equality constraints as $h_1(z)=\cdots=h_s(z)=0$. Then the feasible set is
\begin{equation}
K:=\{z\in\R^{n_0+2n_1+n_2}\mid g_\ell(z)\ge 0\ (\ell=1,\dots,m),\ h_r(z)=0\ (r=1,\dots,s)\}
\label{eq:feasible-set-K}
\end{equation}
and the objective function $c^\top y$ is a linear polynomial in $z$.

Giving a valid upper bound $\gamma$ for the verification problem is equivalent to proving that the polynomial
\begin{equation}
p_\gamma(z):=\gamma-c^\top y
\label{eq:pgamma}
\end{equation}
is nonnegative on $K$. Therefore, the verification problem reduces to the question of how to certify the nonnegativity of a polynomial on a semialgebraic set by a computationally tractable certificate. This is precisely the starting point for applying Lasserre-type moment-SOS hierarchies and Putinar-type Positivstellensätze \citep{Putinar1993,Lasserre2001}. Polynomial and SOS-based approaches to neural network verification are studied, for example, by \citet{NewtonPapachristodoulou2023}.

\subsection{SOS certificates, degree-bounded relaxations, and problem-specific sparse certificates}

A polynomial $q(z)$ is said to be SOS if there exist finitely many polynomials $r_1,\dots,r_N$ such that
\[
q(z)=\sum_{i=1}^N r_i(z)^2.
\]
Since every SOS polynomial is globally nonnegative and SOS feasibility admits a Gram matrix representation,
SOS relaxations can be formulated and solved as semidefinite programs
\citep{Lasserre2001,WakiKimKojimaMuramatsu2006}.
Their dual counterparts are moment relaxations, expressed in terms of moment and localizing matrices.

For the inequality constraints, define the quadratic module
\begin{align}
\qmod(g)&:=\Bigl\{\sigma_0+\sum_{\ell=1}^m \sigma_\ell g_\ell\ \Big|\ \sigma_0,\sigma_1,\dots,\sigma_m\text{ are SOS polynomials}\Bigr\},
\end{align}
and for the equality constraints, define the ideal
\begin{align}
\ideal(h)&:=\Bigl\{\sum_{r=1}^s \tau_r h_r\ \Big|\ \tau_1,\dots,\tau_s\text{ are polynomials}\Bigr\}.
\end{align}
If
\begin{equation}
p_\gamma\in \qmod(g)+\ideal(h),
\label{eq:full-sos-certificate}
\end{equation}
then $p_\gamma(z)\ge 0$ for every $z\in K$, so $\gamma$ is a valid upper bound for problem \eqref{eq:verification-problem}. In this paper, we call an expression of the form \eqref{eq:full-sos-certificate} a certificate for $p_\gamma$.

For computational purposes, we use degree-bounded truncations. Given a relaxation order
$\omega\in\mathbb N$, define
\begin{align}
\qmod_\omega(g)
&:=\Bigl\{\sigma_0+\sum_{\ell=1}^m \sigma_\ell g_\ell\in\qmod(g)
\ \Big|\ \deg(\sigma_0)\le 2\omega,\ \deg(\sigma_\ell g_\ell)\le 2\omega\ (\ell=1,\dots,m)\Bigr\},
\label{eq:truncated-qmodule}\\
\ideal_\omega(h)
&:=\Bigl\{\sum_{r=1}^s \tau_r h_r\in\ideal(h)
\ \Big|\ \deg(\tau_r h_r)\le 2\omega\ (r=1,\dots,s)\Bigr\}
\label{eq:truncated-ideal}
\end{align}
and let
\begin{equation}
\gamma^{(\omega)}:=\inf\{\gamma\in\R\mid p_\gamma\in \qmod_\omega(g)+\ideal_\omega(h)\}.
\label{eq:order-omega-bound}
\end{equation}
Then $\gamma^*\le \gamma^{(\omega)}$ always holds. Moreover, if the feasible set is compact and the quadratic module is Archimedean, Putinar's Positivstellensatz guarantees that $\gamma^{(\omega)}\downarrow\gamma^*$ as $\omega\to\infty$ \citep{Putinar1993,Lasserre2001}.
The number $\omega$ is called the relaxation order.
As $\omega$ increases, higher-degree polynomials are allowed in the SOS multipliers $\sigma_0$, $\sigma_\ell$, and
in the polynomial multipliers $\tau_r$. Consequently, the number of unknown coefficients grows rapidly,
and the resulting certificate search problem becomes increasingly large.
The central question of this paper is to identify conditions under which this bound is already tight at a low relaxation order, in particular at
$\omega=2$.

The relaxation used in the remainder of the paper is not the generic dense SOS relaxation itself but a block-sparse SOS certificate class tailored to the graph-structured sparsity of the ReLU verification problem. This perspective is consistent with the sparse-SOS literature exploiting correlative or term sparsity \citep{WakiKimKojimaMuramatsu2006,WangMagronLasserre2021,WangMagronLasserreMai2022},
but our aim is not to develop a general hierarchy. Instead, in Sections~\ref{sec:graph} and \ref{sec:reduction}
we define a problem-specific certificate class based on the input-sharing graph and its local blocks, and then prove its tightness under
the matching assumption. Thus, the sparse relaxation studied here is a problem-specific block-sparse SOS relaxation
designed for the local geometry of one-layer ReLU verification.

\subsection{The one-unit tightness result as a fundamental component of the analysis}

Within SDP-based neural network verification,
DeepSDP is a representative example \citep{FazlyabMorariPappas2022}, and \citet{Zhang2020} provides a geometric analysis of tightness
for such relaxations. 
The present paper is motivated by the work of \citet{AKY2026},
who proved tightness of semidefinite relaxations for one-layer ReLU networks in three cases:
a single neuron, an ellipsoidal input set, and a diagonal weight matrix over a box input set.
For our analysis, we use the following one-unit consequences of the tightness results of Azuma et al.~\cite{AKY2026}.

Here, a local problem means the problem associated with one connected component,
while a unit means a single ReLU, say $x_j=\relu(v_j)$. Thus, a one-unit block may still have a multidimensional input vector.
\begin{enumerate}[label=\textbf{(A\arabic*)},leftmargin=24mm]
\item \textbf{one-unit interval-input case.}
The local problem contains a single unit, and the effective input reduces to a scalar $s$ over an interval
\[
U=[\underline s,\overline s]\subset \R.
\]
In our notation, the input data consist of a $1\times 1$ matrix $W\in\R^{1\times 1}$, a scalar $b\in\R$,
and the interval input set $U$, and the feasible set has the form
\[
v=Ws+b,\qquad x=\relu(v),\qquad s\in U.
\]

\item \textbf{ellipsoidal-input case.}
The local problem involves a single unit, and the input set is an ellipsoid
\[
U=\{u\in\R^m\mid (u-u^{\mathrm c})^\top Q^{-1}(u-u^{\mathrm c})\le 1\}
\]
with $Q\succ 0$, while the pre-activation is given by the affine function
\[
v=w^\top u+b.
\]

\item \textbf{one-unit rectangular-input case.}
The local problem contains a single unit, and the input set is a rectangle
\[
U=[\underline u_1,\overline u_1]\times\cdots\times[\underline u_m,\overline u_m]\subset\R^m.
\]
The pre-activation is affine in the input:
\[
v=w^\top u+b,
\]
This is the form to which isolated-vertex blocks and the reduced one-unit problems introduced later
are reduced. The interval-input case \textbf{(A1)} is the special case $m=1$ of the rectangular-input setting.
\end{enumerate}

These results will be applied to isolated vertices and to the reduced one-unit problems arising from an edge block. The new difficulty is to connect the resulting one-unit certificates across the shared scalar.

%% file: section-03.tex
\section{Preprocessing, Undecided ReLUs, and the Input-Sharing Graph}\label{sec:graph}

In this section, we introduce the basic notation for the block-sparse SOS relaxation defined in Section~\ref{sec:reduction}, reflecting the structure of the one-layer ReLU verification problem studied in this paper.
 We first remove the determined ReLUs by interval evaluation and derive a reduced problem containing only
 undecided ReLUs. We then define the graph $H=(\Uset,E)$ that records input sharing among undecided ReLUs, and introduce the matching condition and the
 shared rank $q_e$. Finally, we define regular rank-one edges, which will be used in Section~\ref{sec:edge}, and clarify what decomposes componentwise
 under the matching assumption.

\subsection{Elimination of determined ReLUs and the reduced verification problem}

Assume that for each hidden unit $j=1,\dots,n_1$, preprocessing provides an interval bound
\begin{equation}
[\underline v_j,\overline v_j]
\end{equation}
for the pre-activation $v_j$. We then partition the unit indices as
\begin{align}
\mathcal J^- &:= \{j\in\{1,\dots,n_1\}\mid \overline v_j\le 0\},\\
\mathcal J^+ &:= \{j\in\{1,\dots,n_1\}\mid \underline v_j\ge 0\},\\
\Uset &:= \{j\in\{1,\dots,n_1\}\mid \underline v_j<0<\overline v_j\}.
\label{eq:stable-unstable-partition-sec3}
\end{align}
Here $\mathcal J^-\cup\mathcal J^+$ is the set of determined ReLUs and $\Uset$ is the set of undecided ReLUs. In the verification literature, the same partition is often described as stable versus unstable ReLUs. In this paper, however, we consistently use the terminology ``determined ReLUs'' and
``undecided ReLUs.'' For $j\in\mathcal J^-$
the ReLU output is fixed at 
$x_j=0$, whereas for $j\in\mathcal J^+$, it is given by $x_j=v_j$. Thus, all essential nonlinearity is confined to $\Uset$.

Let $W_{\Uset,:}$ denote the submatrix of $W$ consisting of the rows indexed by $\Uset$, and let $A_{:,\Uset}$ denote the submatrix of $A$ consisting of the columns indexed by $\Uset$. Define
\begin{align}
\bar W &:= W_{\Uset,:},
&
\bar b &:= b_{\Uset},\\
\bar A &:= A_{:,\Uset},
&
\bar B &:= A_{:,\mathcal J^+}W_{\mathcal J^+,:},
&
\bar a &:= a + A_{:,\mathcal J^+}b_{\mathcal J^+}.
\end{align}
Then, the verification problem \eqref{eq:verification-problem} can be rewritten equivalently as
\begin{equation}
\gamma^*=
\max\Bigl\{
 c^\top(\bar A x+\bar B u+\bar a)
 \ \Big|\
 u\in\mathcal U,
 \ v=\bar W u+\bar b,
 \ x=\relu(v)
\Bigr\}.
\label{eq:reduced-verification-problem-sec3}
\end{equation}
Here the vectors $v$ and $x$ are indexed only by $\Uset$. Therefore, in the subsequent discussion,
 the essential semialgebraic constraints are simply
\begin{equation}
 v_j = \bar W_{j:}u+\bar b_j,
 \qquad
 x_j\ge 0,
 \qquad
 x_j-v_j\ge 0,
 \qquad
 x_j(x_j-v_j)=0
 \qquad (j\in\Uset).
\label{eq:reduced-semialgebraic-system-sec3}
\end{equation}

Using the notation of Section~\ref{sec:prelim}, let
\begin{equation}
 p_\gamma(u,v,x):=\gamma-c^\top(\bar A x+\bar B u+\bar a).
\label{eq:reduced-pgamma-sec3}
\end{equation}
Our aim is to understand when the nonnegativity of this $p_\gamma$ can be captured tightly by SOS certificates
compatible with a sparse block structure.
This requires an explicit description of the input coordinates shared by the undecided ReLUs in $\Uset$.

\subsection{Input supports, input-sharing graph, matching, and shared rank}

For each undecided ReLU $j\in\Uset$, define the set of input coordinates on which its pre-activation depends by
\begin{equation}
S_j:=\supp(\bar W_{j:})
=\{k\in\{1,\dots,n_0\}\mid \bar W_{jk}\neq 0\}.
\label{eq:input-support-sec3}
\end{equation}
Thus, the ReLU constraint for unit $j$ depends not on the entire input vector $u$ but only on the subvector $u_{S_j}$.

Based on these input supports, we introduce an undirected graph describing the interaction among undecided ReLUs.

\begin{definition}[input-sharing graph]
For distinct vertices $i,j\in\Uset$, define
\begin{equation}
\{i,j\}\in E
\quad\Longleftrightarrow\quad
S_i\cap S_j\neq\emptyset.
\label{eq:edge-definition-sec3}
\end{equation}
Then, $H=(\Uset,E)$ is called the input-sharing graph of the undecided ReLUs.
\end{definition}

Condition \eqref{eq:edge-definition-sec3} indicates whether two undecided ReLUs are coupled through common input coordinates.
Therefore, the graph $H$ identifies where the essential interactions appear in a global block-sparse SOS certificate.
In particular, when $H$ is edgeless, no undecided ReLU shares inputs with any other undecided ReLU.

For an edge $e=\{i,j\}\in E$, we write the shared input set as
\begin{equation}
S_e:=S_i\cap S_j.
\label{eq:shared-index-set-sec3}
\end{equation}
We also impose the matching condition as the main graph-theoretic assumption.

\begin{definition}[matching]
The input-sharing graph $H=(\Uset,E)$ is called a \emph{matching} if
\begin{equation}
\deg_H(j)\le 1
\qquad (\forall j\in\Uset)
\label{eq:matching-definition-sec3}
\end{equation}
holds.
\end{definition}
Here, $\deg_H(j)$ is the number of other undecided ReLUs that share at least one input coordinate with unit $j$:
\[
\deg_H(j)=\left|\{i\in\Uset\setminus\{j\}\mid S_i\cap S_j\neq\emptyset\}\right|.
\]
Thus, the matching condition means that each undecided ReLU shares input variables with at most one other undecided ReLU.

Under the matching assumption, every connected component is either an isolated vertex or a single edge.
This allows compatibility and assembly to be treated component by component, rather than over the entire graph at once. Accordingly, the matching assumption is a central hypothesis in the results of Sections~4 and~5.
%

To measure the effective dimension of the shared part of each edge, we introduce the shared rank.

\begin{definition}[shared rank]
For an edge $e=\{i,j\}\in E$, define
\begin{equation}
q_e
:=
\dim \operatorname{span}\{\bar W_{i,S_e}^{\top},\bar W_{j,S_e}^{\top}\}.
\label{eq:qe-definition-sec3}
\end{equation}
We call $q_e$ the shared rank of $e$.
\end{definition}

Here $\bar W_{i,S_e}$ and $\bar W_{j,S_e}$ are the coefficient vectors obtained by restricting rows $i$ and $j$ to the shared input set $S_e$. Thus $q_e$ measures the number of effective directions through which the two undecided units depend affinely on the shared inputs. In particular, when
\begin{equation}
q_e\le 1,
\label{eq:qe-le-one-sec3}
\end{equation}
the reduction leaves only a single common scalar to represent the shared part.
This condition is the starting point for the order-$2$ local tightness result for two-unit blocks in Section~\ref{sec:edge}.

\subsection{Regular rank-one edges and the componentwise decomposition under matching}

For a subset $I\subseteq\{1,\dots,n_0\}$, write the projection of the input box onto that coordinate subspace as
\begin{equation}
\mathcal U_I:=\{u_I\in\R^I\mid \underline u_I\le u_I\le \bar u_I\}.
\label{eq:projected-input-box-sec3}
\end{equation}
Fix an edge $e=\{i,j\}\in E$. For each $h\in\{i,j\}$, define the private index set by
\begin{equation}
P_h:=S_h\setminus S_e.
\label{eq:private-index-set-sec3}
\end{equation}
Then the pre-activation decomposes as
\begin{equation}
v_h=\bar W_{h,P_h}u_{P_h}+\bar W_{h,S_e}u_{S_e}+\bar b_h
\qquad (h=i,j).
\label{eq:edgewise-affine-split-sec3}
\end{equation}

Based on this decomposition, define for each $h\in\{i,j\}$ the bounds of the private and shared parts by
\begin{align}
\underline p_{h,e}
&:=
\min_{u_{P_h}\in\mathcal U_{P_h}}
\bigl(\bar W_{h,P_h}u_{P_h}+\bar b_h\bigr),
&
\overline p_{h,e}
&:=
\max_{u_{P_h}\in\mathcal U_{P_h}}
\bigl(\bar W_{h,P_h}u_{P_h}+\bar b_h\bigr),
\label{eq:plan-private-bounds-sec3}\\
\underline s_{h,e}
&:=
\min_{u_{S_e}\in\mathcal U_{S_e}}
\bar W_{h,S_e}u_{S_e},
&
\overline s_{h,e}
&:=
\max_{u_{S_e}\in\mathcal U_{S_e}}
\bar W_{h,S_e}u_{S_e}.
\label{eq:plan-shared-bounds-sec3}
\end{align}

\begin{definition}[regular rank-one edge]
An edge $e=\{i,j\}\in E$ is called a \emph{regular rank-one edge} if the following two conditions hold:
\begin{enumerate}[label=\rmfamily(M3\alph*),leftmargin=28pt]
\item \label{item:M3a-sec3}
\textbf{rank-one shared direction:}
\begin{equation}
q_e\le 1.
\label{eq:M3a-sec3}
\end{equation}

\item \label{item:M3b-sec3}
\textbf{edge regularity:}
\begin{equation}
\underline p_{h,e}+\overline s_{h,e}<0<\overline p_{h,e}+\underline s_{h,e}
\qquad (h=i,j).
\label{eq:M3b-sec3}
\end{equation}
\end{enumerate}
\end{definition}

Condition \eqref{eq:M3a-sec3} ensures
 that the shared part is represented by a single scalar.   
 Condition \eqref{eq:M3b-sec3} states
  that, 
  regardless of how the shared variable
  $u_{S_e}$ is fixed within its admissible range,
  the private variables can still make the pre-activation of each unit take both negative and positive values.
  Thus a regular rank-one edge is a local condition expressing that the shared structure is effectively
  one-dimensional and that both units remain undecided throughout the entire edge.

\begin{proposition}[operational meaning of edge regularity]
\label{prop:sec3-edge-regularity-operational}
Fix an edge $e=\{i,j\}\in E$ and $h\in\{i,j\}$.
Then the following statements are equivalent:
\begin{enumerate}[label=\rmfamily(\alph*),leftmargin=28pt]
\item $\underline p_{h,e}+\overline s_{h,e}<0<\overline p_{h,e}+\underline s_{h,e}$ holds;
\item for every $u_{S_e}\in\mathcal U_{S_e}$, there exist $u_{P_h}^-,u_{P_h}^+\in\mathcal U_{P_h}$ such that
    \begin{equation*}
    \bar W_{h,P_h}u_{P_h}^-+\bar W_{h,S_e}u_{S_e}+\bar b_h<0,
    \qquad
    \bar W_{h,P_h}u_{P_h}^++\bar W_{h,S_e}u_{S_e}+\bar b_h>0.
    \end{equation*}
    In other words, even after the shared part is fixed, unit $h$ can still attain both negative and positive pre-activation values by varying its private part.
\end{enumerate}
\end{proposition}

\begin{proof}
Fix $u_{S_e}\in\mathcal U_{S_e}$. Then, the shared term $\bar W_{h,S_e}u_{S_e}$ is a point in the interval $[\underline s_{h,e},\overline s_{h,e}]$.
On the other hand, as $u_{P_h}$ varies over $\mathcal U_{P_h}$, the private term $\bar W_{h,P_h}u_{P_h}+\bar b_h$ ranges over the entire
interval $[\underline p_{h,e},\overline p_{h,e}]$. Hence, for the fixed shared term, the condition that the pre-activation $v_h$ can attain both negative and positive values is equivalent to
\begin{equation*}
\underline p_{h,e}+\bar W_{h,S_e}u_{S_e}<0<\overline p_{h,e}+\bar W_{h,S_e}u_{S_e}.
\end{equation*}
This requirement must hold uniformly for all
$u_{S_e}\in\mathcal U_{S_e}$, and
equivalently, it is characterized by the following worst-case inequality.
\begin{equation*}
\underline p_{h,e}+\overline s_{h,e}<0<\overline p_{h,e}+\underline s_{h,e}.
\end{equation*}
\end{proof}

Proposition~\ref{prop:sec3-edge-regularity-operational} provides
 a direct interpretation of the two requirements: condition \eqref{eq:M3a-sec3} makes the shared dependence one-dimensional, while condition \eqref{eq:M3b-sec3} ensures that each unit can attain both signs after the shared input is fixed. This excludes cases in which the shared coupling effectively determines one of the units.

Indeed, if $P_h=\emptyset$, so that no private variables are present, then condition \eqref{eq:M3b-sec3} requires the shared part itself to cross zero. Conversely, if $\bar W_{h,S_e}=0$, so that the shared part does not affect unit $h$, then \eqref{eq:M3a-sec3} is automatically satisfied, and the analysis reduces to a one-unit problem involving only the private variables.

Finally, to connect the above discussion with the block-sparse SOS certificates defined in Section~\ref{sec:reduction},
for each undecided ReLU $j\in\Uset$,
 we define the one-unit local variable tuple by
\begin{equation}
\zeta_j := (u_{S_j},v_j,x_j)
\label{eq:vertex-local-block-sec3}
\end{equation}
and for each edge $e=\{i,j\}\in E$ we define the two-unit local variable tuple by
\begin{equation}
\zeta_e := (u_{S_i\cup S_j},v_i,x_i,v_j,x_j).
\label{eq:edge-local-block-sec3}
\end{equation}
These local tuples serve as the blocks in the subsequent discussion.
Section~4 assembles the componentwise certificates, and Section~5 proves the edge-local exactness required for regular rank-one edges.

%% file: section-04.tex
\section{Component Decomposition under the Matching Assumption}\label{sec:reduction}

In this section, we study the case in which the input-sharing graph $H=(\Uset,E)$ introduced in Section~\ref{sec:graph} is a matching.
Under this assumption, all nonlinear coupling is confined to isolated-vertex components and single-edge components. This structure allows us to decompose both the feasible set and the objective, and to prove an assembly result that converts local order-$2$ certificates into a global certificate.

\subsection{Component decomposition induced by the matching assumption}

We first note that, under the matching assumption, every connected component has at most two vertices.

\begin{proposition}\label{prop:sec4-components-matching}
Assume that the input-sharing graph $H=(\Uset,E)$ is a matching. Then,
 every connected component of $H$ is either an isolated vertex or a single edge.
\end{proposition}

\begin{proof}
By definition of a matching, each vertex $j\in\Uset$ is incident to at most one edge. Hence, if some connected component contained three or more vertices, an intermediate vertex in that component would have degree at least $2$, which is impossible. Therefore every connected component is either a singleton component of degree $0$ or a two-vertex, one-edge component.
\end{proof}

Proposition~\ref{prop:sec4-components-matching} implies that all remaining nonconvex coupling in the global problem is localized
within one-unit blocks and two-unit edge blocks. We next specify the input coordinates assigned to each block and use this structure
to decompose both the feasible set and the objective polynomial.

Let the set of isolated vertices be
\begin{equation}
\mathcal I:=\{j\in\Uset\mid \deg_H(j)=0\}.
\label{eq:sec4-isolated-set}
\end{equation}
For each input coordinate $k\in\{1,\dots,n_0\}$, define
\begin{equation}
\Lambda(k):=\{j\in\Uset\mid k\in S_j\}.
\label{eq:sec4-Lambdak}
\end{equation}
Under the matching assumption, one has $|\Lambda(k)|\le 2$ for every $k$. Indeed, if $|\Lambda(k)|\ge 3$, then those three vertices would be pairwise
adjacent and hence form a triangle, contradicting the matching assumption. Therefore,
 each input coordinate belongs to exactly one of the following sets:
\begin{alignat}{2}
T_j &:=\{k\mid \Lambda(k)=\{j\}\} \quad & & (j\in\Uset), \label{eq:sec4-Tj} \\
T_e &:=\{k\mid \Lambda(k)=\{i,j\}\} \quad  & & (e=\{i,j\}\in E), \label{eq:sec4-Te}
\end{alignat}
and
\begin{equation}
T_0:=\{1,\dots,n_0\}\setminus \bigcup_{j\in\Uset}S_j.
\label{eq:sec4-T0}
\end{equation}
That is,
\begin{equation}
\{1,\dots,n_0\}
=
T_0
\sqcup
\Bigl(\bigsqcup_{j\in\Uset}T_j\Bigr)
\sqcup
\Bigl(\bigsqcup_{e\in E}T_e\Bigr).
\label{eq:sec4-inputpartition}
\end{equation}


The input coordinates in $T_0$ do not appear in the nonlinear constraints of the undecided ReLUs. They appear only in the box constraints
\[
\underline u_k\le u_k\le \bar u_k
\qquad (k\in T_0)
\]
and in the linear objective term $-\sum_{k\in T_0}\delta_k u_k$. Thus $u_{T_0}$ forms a free linear block that shares neither nonlinear equalities nor ReLU constraints with any other local block. Accordingly, when $T_0\neq\emptyset$, if $\Uset=\emptyset$, the reduced problem is simply a linear optimization problem over the input box and can be certified by the box-affine argument in Lemma~\ref{lem:sec5-box-affine-certificate}. Otherwise, we fix one connected component $C^\circ$ of the matching and absorb $u_{T_0}$ together with the linear term $-\sum_{k\in T_0}\delta_k u_k$ into the local block of that component. Equivalently, the variables, box generators, and local affine objective of $C^\circ$ are augmented to include $u_{T_0}$, the bounds $\underline u_k\le u_k\le \bar u_k$ for $k\in T_0$, and the term $-\sum_{k\in T_0}\delta_k u_k$. Since this absorption creates no new variable sharing with the other components, the direct-product structure of the reduced problem is preserved.

\subsection{Componentwise decomposition of the objective polynomial and local optimal values}

For the objective polynomial of the reduced problem,
\begin{equation}
p_\gamma(u,v,x):=\gamma-c^\top(\bar A x+\bar B u+\bar a),
\label{eq:sec4-pgamma}
\end{equation}
define
\begin{equation}
\eta_\gamma:=\gamma-c^\top\bar a,
\qquad
\alpha_j:=(\bar A^\top c)_j\quad (j\in\Uset),
\qquad
\delta_k:=(\bar B^\top c)_k\quad (k=1,\dots,n_0).
\label{eq:sec4-alpha-delta}
\end{equation}
Then
\begin{equation}
p_\gamma
=
\eta_\gamma
-
\sum_{j\in\Uset}\alpha_j x_j
-
\sum_{k=1}^{n_0}\delta_k u_k.
\label{eq:sec4-pgamma-linearform}
\end{equation}
For each isolated vertex $j\in\mathcal I$, define
\begin{equation}
\Theta_j^{\mathrm{lin}}(\zeta_j)
:=
-\alpha_j x_j
- \sum_{k\in T_j}\delta_k u_k,
\label{eq:sec4-component-vertex}
\end{equation}
where $\zeta_j:=(u_{T_j},v_j,x_j)$. For
each edge component $e=\{i,j\}\in E$, define
\begin{equation}
\Theta_e^{\mathrm{lin}}(\zeta_e)
:=
-\alpha_i x_i-\alpha_j x_j
-
\sum_{k\in T_i\cup T_e\cup T_j}\delta_k u_k,
\label{eq:sec4-component-edge}
\end{equation}
where $\zeta_e:=(u_{T_i\cup T_e\cup T_j},v_i,x_i,v_j,x_j)$.
If the base component $C^\circ$ absorbs $T_0$, then the corresponding function
 $\Theta_{C^\circ}^{\mathrm{lin}}$ and local block
are enlarged to include $-\sum_{k\in T_0}\delta_k u_k$ and $u_{T_0}$.
Equivalently, the local box generators of $C^\circ$ are augmented by
 $u_k-\underline u_k$ and $\bar u_k-u_k$ for $k\in T_0$. Under this convention, distinct connected components in the matching
 share no variables. Therefore,
\begin{equation}
p_\gamma
=
\eta_\gamma
+
\sum_{e\in E}\Theta_e^{\mathrm{lin}}(\zeta_e)
+
\sum_{j\in\mathcal I}\Theta_j^{\mathrm{lin}}(\zeta_j).
\label{eq:sec4-exact-linear-decomp}
\end{equation}

We next define the local feasible sets corresponding to each component. For an isolated vertex $j\in\mathcal I$, let
\begin{align}
g_j(\zeta_j)&:=\bigl\{u_k-\underline u_k,\ \bar u_k-u_k\ (k\in T_j),\ x_j,\ x_j-v_j\bigr\},
\label{eq:sec4-gj}\\
h_j(\zeta_j)&:=\bigl\{v_j-\bar W_{j,T_j}u_{T_j}-\bar b_j,\ x_j(x_j-v_j)\bigr\},
\label{eq:sec4-hj}
\end{align}
and for an edge component $e=\{i,j\}\in E$, let
\begin{align}
g_e(\zeta_e)
&:=\bigl\{u_k-\underline u_k,\ \bar u_k-u_k\ (k\in T_i\cup T_e\cup T_j),\ x_i,\ x_i-v_i,\ x_j,\ x_j-v_j\bigr\},
\label{eq:sec4-ge}\\
h_e(\zeta_e)
&:=\Bigl\{v_i-\bar W_{i,T_i}u_{T_i}-\bar W_{i,T_e}u_{T_e}-\bar b_i,\\
&\qquad\ v_j-\bar W_{j,T_j}u_{T_j}-\bar W_{j,T_e}u_{T_e}-\bar b_j,\\
&\qquad\ x_i(x_i-v_i),\ x_j(x_j-v_j)\Bigr\}.
\label{eq:sec4-he}
\end{align}
We write the corresponding local feasible sets as
\begin{equation}
K_j:=\{\zeta_j\mid g_j(\zeta_j)\ge 0,\ h_j(\zeta_j)=0\}
\qquad (j\in\mathcal I),
\label{eq:sec4-Kj}
\end{equation}
\begin{equation}
K_e:=\{\zeta_e\mid g_e(\zeta_e)\ge 0,\ h_e(\zeta_e)=0\}
\qquad (e\in E).
\label{eq:sec4-Ke}
\end{equation}
Define the corresponding local optimal values by
\begin{equation}
m_j:=\min\{\Theta_j^{\mathrm{lin}}(\zeta_j)\mid \zeta_j\in K_j\}
\qquad (j\in\mathcal I),
\label{eq:sec4-localopt-vertex}
\end{equation}
\begin{equation}
m_e:=\min\{\Theta_e^{\mathrm{lin}}(\zeta_e)\mid \zeta_e\in K_e\}
\qquad (e\in E).
\label{eq:sec4-localopt-edge}
\end{equation}

\begin{proposition}[value decomposition under the matching assumption]
\label{prop:sec4-value-decomposition}
Assume that the input-sharing graph $H=(\Uset,E)$ is a matching. Then,
the feasible set of the reduced problem decomposes as a direct product of the local feasible sets of
the isolated-vertex and single-edge components. Consequently, for any $\gamma\in\R$ one has
\begin{equation}
\min p_\gamma
=
\eta_\gamma
+
\sum_{e\in E}m_e
+
\sum_{j\in\mathcal I}m_j.
\label{eq:sec4-global-value-decomp}
\end{equation}
\end{proposition}

\begin{proof}
The key point is that, under the matching assumption, both the constraints and the objective decompose
according to the same component partition.
Indeed, by Proposition~\ref{prop:sec4-components-matching}, each connected component is either an isolated vertex or a single edge,
and by \eqref{eq:sec4-inputpartition}, each input coordinate is assigned to exactly one component. Moreover, as mentioned above,
the variables in $T_0$ appear only in box constraints and linear objective terms. Absorbing them into the base component $C^\circ$ therefore creates no coupling with the other components.
Therefore, the global feasible set of the reduced problem can be identified with
\[
\prod_{e\in E}K_e\times \prod_{j\in\mathcal I}K_j.
\]

On the other hand, \eqref{eq:sec4-exact-linear-decomp} shows that the objective polynomial decomposes exactly into the constant term $\eta_\gamma$ and the sum of the local objective functions. Hence, minimizing over a direct-product set separates into componentwise minimizations:
\[
\min p_\gamma
=
\eta_\gamma
+
\sum_{e\in E}\min_{\zeta_e\in K_e}\Theta_e^{\mathrm{lin}}(\zeta_e)
+
\sum_{j\in\mathcal I}\min_{\zeta_j\in K_j}\Theta_j^{\mathrm{lin}}(\zeta_j).
\]
By the definitions \eqref{eq:sec4-localopt-edge} and \eqref{eq:sec4-localopt-vertex}, these minima are exactly $m_e$ and $m_j$, which yields \eqref{eq:sec4-global-value-decomp}.
\end{proof}

\subsection{Assembly from local tightness to global tightness}

At this stage, the global assembly argument required for the paper consists of the following three components.
\begin{enumerate}[label=(\roman*),leftmargin=2.3em]
\item the feasible set decomposes into a direct product of local feasible sets,
\item the objective polynomial decomposes into the sum of local affine objective functions, and
\item if each local certificate belongs to the truncated quadratic module/ideal of degree $2$, then any finite sum of them yields a global certificate at the same degree.
\end{enumerate}
Thus, this section isolates the assembly details needed to combine the local results, so that the remaining mathematical task is to prove order-$2$ tightness for a single two-unit edge component.

Define
\begin{equation}
\mathcal C_{\mathrm{sp}}^{(2)}
:=
\left\{
\sum_{e\in E}Q_e(\zeta_e)+\sum_{j\in\mathcal I}P_j(\zeta_j)
\ \middle|\
Q_e\in \qmod_2(g_e)+\ideal_2(h_e),\
P_j\in \qmod_2(g_j)+\ideal_2(h_j)
\right\}.
\label{eq:sec4-sparsecone}
\end{equation}
This defines the order-$2$ certificate class for the block-sparse SOS relaxation induced by the connected-component
decomposition of the matching.
The corresponding relaxation value is denoted by
\begin{equation}
\gamma_{\mathrm{sp}}^{(2)}
:=
\inf\bigl\{\gamma\in\R\mid p_\gamma\text{ admits a certificate belonging to }\mathcal C_{\mathrm{sp}}^{(2)}\bigr\}.
\label{eq:sec4-gamma-sp}
\end{equation}

\begin{proposition}[global tightness criterion under the matching assumption]
\label{prop:sec4-global-exactness}
Assume that the input-sharing graph $H=(\Uset,E)$ is a matching. Suppose further that for every isolated vertex $j\in\mathcal I$ and every edge $e\in E$,
\begin{equation}
\Theta_j^{\mathrm{lin}}-m_j
\in
\qmod_2(g_j)+\ideal_2(h_j)
\qquad (j\in\mathcal I),
\label{eq:sec4-local-cert-vertex}
\end{equation}
and
\begin{equation}
\Theta_e^{\mathrm{lin}}-m_e
\in
\qmod_2(g_e)+\ideal_2(h_e)
\qquad (e\in E)
\label{eq:sec4-local-cert-edge}
\end{equation}
hold. Then, the order-$2$ block-sparse SOS relaxation is tight, namely,
\begin{equation}
\gamma_{\mathrm{sp}}^{(2)}=\gamma^*.
\label{eq:sec4-exactness-sp}
\end{equation}
\end{proposition}

\begin{proof}
Applying Proposition~\ref{prop:sec4-value-decomposition} with $\gamma=\gamma^*$, and using the fact that the global minimum of $p_{\gamma^*}$ is $0$, we obtain
\begin{equation}
\eta_{\gamma^*}+\sum_{e\in E}m_e+\sum_{j\in\mathcal I}m_j=0.
\label{eq:sec4-balance-at-opt}
\end{equation}
On the other hand, summing the assumptions \eqref{eq:sec4-local-cert-vertex} and \eqref{eq:sec4-local-cert-edge} gives
\begin{align*}
&\sum_{e\in E}\bigl(\Theta_e^{\mathrm{lin}}-m_e\bigr)
+\sum_{j\in\mathcal I}\bigl(\Theta_j^{\mathrm{lin}}-m_j\bigr)\\
&\qquad=
\Bigl(\sum_{e\in E}\Theta_e^{\mathrm{lin}}+\sum_{j\in\mathcal I}\Theta_j^{\mathrm{lin}}\Bigr)
-\sum_{e\in E}m_e-\sum_{j\in\mathcal I}m_j\\
&\qquad=
p_{\gamma^*}-\eta_{\gamma^*}-\sum_{e\in E}m_e-\sum_{j\in\mathcal I}m_j
=
p_{\gamma^*},
\end{align*}
where the last equality follows from \eqref{eq:sec4-balance-at-opt}.
Hence, $p_{\gamma^*}\in \mathcal C_{\mathrm{sp}}^{(2)}$, and therefore $\gamma_{\mathrm{sp}}^{(2)}\le \gamma^*$.

For the reverse inequality, each summand in \eqref{eq:sec4-sparsecone} is nonnegative on the corresponding local feasible set,
and therefore their sum is also nonnegative on the global feasible set.
Indeed, if $p_\gamma\in \mathcal C_{\mathrm{sp}}^{(2)}$, then $p_\gamma$ is nonnegative on the feasible set of the reduced problem, so $\gamma$ is a valid upper bound for the original verification problem. Thus $\gamma\ge \gamma^*$, and by the definition \eqref{eq:sec4-gamma-sp} we obtain $\gamma^*\le \gamma_{\mathrm{sp}}^{(2)}$. This proves \eqref{eq:sec4-exactness-sp}.
\end{proof}

This completes the global assembly step. It remains to
 prove the edge-local certificate under the regular rank-one condition. Isolated vertices are handled by the one-unit results of Azuma et al.~\cite{AKY2026}.

%% file: section-05.tex
\section{Local Exactness for Regular Rank-One Edges and the Main Theorem for Matching}\label{sec:edge}

In this section, we establish two-unit local tightness for the
only nontrivial component that remains after the component decomposition of Section~\ref{sec:reduction}, namely a single-edge component.
Our goal is to prove that for every edge-local affine objective $\ell_e$ expressible in the rank-one reduced variables,
\[
\ell_e-\gamma_e^\star\in \qmod_2(g_e)+\ideal_2(h_e)
\]
holds. We then combine this edge-local certificate
 with the global assembly result of Section~\ref{sec:reduction} to obtain the main tightness
 theorem for matching input-sharing graphs.

The argument proceeds in four steps: reduce the shared part to a scalar, describe the one-unit and two-unit hulls using endpoint slices, convert this hull description into an order-$2$ certificate using the result of Azuma et al.~\cite{AKY2026}
 and Farkas' lemma, and finally combine the local result with the assembly criterion of Section~\ref{sec:reduction}.

Recall that \textbf{(M3a)} provides a one-dimensional shared scalar,
while \textbf{(M3b)} ensures that both units remain undecided after that scalar is fixed.
Fix an edge $e=\{i,j\}\in E$. We first describe the geometry in affinely reduced coordinates and then
translate the resulting certificate back to the original edge block.

\subsection{An edge-local normal form induced by regular rank-one edges}

For each edge $e=\{i,j\}$, recall the shared input set introduced in Section~3,
\[
S_e:=S_i\cap S_j.
\]
Condition \textbf{(M3a)} allows us to reduce this shared part to a single scalar.

By the assumption $q_e\le 1$,
\[
q_e=
\dim \operatorname{span}\{\bar W_{i,S_e}^{\top},\bar W_{j,S_e}^{\top}\}\le 1,
\]
so there exist a nonzero vector $r_e\in \R^{S_e}$ and scalars $\alpha_i,\alpha_j$ such that
\begin{equation}
\bar W_{i,S_e}=\alpha_i r_e^{\top},
\qquad
\bar W_{j,S_e}=\alpha_j r_e^{\top}.
\label{eq:sec5-collinear}
\end{equation}
Here $r_e$ is a reference direction representing the common input direction
through which the two units depend on the shared part.

For each $h\in\{i,j\}$, split the pre-activation into its shared and nonshared parts:
\[
v_h=\bar W_{h,S_e}u_{S_e}+\bar W_{h,S_h\setminus S_e}u_{S_h\setminus S_e}+\bar b_h.
\]
Introduce the scalar obtained from the shared part by
\begin{equation}
t_0:=r_e^{\top}u_{S_e}
\label{eq:sec5-shared-scalar}
\end{equation}
and call the remaining term
\[
p_h:=\bar W_{h,S_h\setminus S_e}u_{S_h\setminus S_e}+\bar b_h
\qquad (h=i,j)
\]
the private affine term of unit $h$. This quantity depends only on input coordinates outside the shared part.
Since it is the affine image of a box, it ranges over some interval $[A_h,B_h]$. Hence,
 the pre-activation can be written as
\begin{equation}
v_h=p_h+\alpha_h t_0,
\qquad
p_h\in [A_h,B_h]
\qquad (h=i,j).
\label{eq:sec5-normal-form-pre}
\end{equation}

Write the range of the shared scalar $t_0$ as
\[
I_e:=\{r_e^{\top}u_{S_e}\mid \underline u_{S_e}\le u_{S_e}\le \bar u_{S_e}\}=[\underline\tau_e,\overline\tau_e].
\]
If $\underline\tau_e=\overline\tau_e$, then the shared scalar is fixed and the term $\alpha_h t_0$
in \eqref{eq:sec5-normal-form-pre} can be absorbed into the private term as a constant. In that case
 the edge block degenerates into the direct product of two one-unit blocks,
 hence no genuine
 two-unit analysis is needed. Therefore, for the nontrivial case, we assume $\underline\tau_e<\overline\tau_e$
 and normalize by
\[
t:=\frac{2t_0-(\underline\tau_e+\overline\tau_e)}{\overline\tau_e-\underline\tau_e}.
\]
To simplify notation, we write $z_h$ for the pre-activation in these rescaled coordinates and obtain
\begin{equation}
z_h=p_h+\alpha_h t,
\qquad
p_h\in [A_h,B_h]
\qquad (h=i,j).
\label{eq:sec5-normal-form}
\end{equation}
Moreover, by normalization,
\begin{equation}
t\in[-1,1].
\label{eq:sec5-tbox}
\end{equation}
Thus the affine change of variables maps $t_0=\underline\tau_e$ to $t=-1$ and $t_0=\overline\tau_e$ to $t=1$. In passing from \eqref{eq:sec5-normal-form-pre} to \eqref{eq:sec5-normal-form}, the constant part $\alpha_h(\underline\tau_e+\overline\tau_e)/2$ is absorbed into the private term and the coefficient of $t$ is renamed $\alpha_h$.
The degenerate case is treated separately in Theorem~\ref{thm:sec5-edge-degree2}.

Let
\begin{equation}
L_h(t):=A_h+\alpha_h t,
\qquad
U_h(t):=B_h+\alpha_h t
\qquad (h=i,j).
\label{eq:sec5-LU}
\end{equation}
Then, the statement that both units remain undecided at the endpoints $t=\pm 1$ can be written as
\begin{equation}
L_h(-1)<0,
\quad
L_h(1)<0,
\quad
U_h(-1)>0,
\quad
U_h(1)>0
\qquad (h=i,j).
\label{eq:sec5-generic-unstable}
\end{equation}

\begin{lemma}[equivalence between edge regularity and endpoint conditions]
\label{lem:sec5-regularity-endpoints}
Assume that $e=\{i,j\}$ is a regular rank-one edge in the sense of Section~\ref{sec:graph}. Then one of the following holds.
\begin{enumerate}[label=\rmfamily(\alph*),leftmargin=28pt]
\item The range $I_e=[\underline\tau_e,\overline\tau_e]$ of the shared scalar is degenerate, that is, $\underline\tau_e=\overline\tau_e$, and the edge block degenerates into the direct product of two one-unit blocks.
\item One has $\underline\tau_e<\overline\tau_e$, and after choosing the normalization in \eqref{eq:sec5-normal-form}--\eqref{eq:sec5-tbox} appropriately, condition \eqref{eq:sec5-generic-unstable} holds.
\end{enumerate}
Conversely, in the nondegenerate case, if \eqref{eq:sec5-generic-unstable} holds under some normalization, then the edge-regularity condition \eqref{eq:M3b-sec3} holds.
\end{lemma}

\begin{proof}
If $I_e=[\underline\tau_e,\overline\tau_e]$ is degenerate with $\underline\tau_e=\overline\tau_e$, then the shared scalar is fixed, and by absorbing $\alpha_h t_0$ into the private term, the edge block degenerates into the direct product of two one-unit blocks. Thus (a) follows.

We now consider the nondegenerate case $\underline\tau_e<\overline\tau_e$.
By \eqref{eq:M3a-sec3}, the shared part reduces to a one-dimensional scalar, and the pre-activations take the form \eqref{eq:sec5-normal-form}. If we normalize so that the center of the affine image of the shared part is absorbed into the private term and only its signed length remains in $\alpha_h t$, then \eqref{eq:M3b-sec3} is equivalent to
\[
A_h+|\alpha_h|<0< B_h-|\alpha_h|
\qquad (h=i,j).
\]
On the other hand, by \eqref{eq:sec5-LU},
\[
\max\{L_h(-1),L_h(1)\}=A_h+|\alpha_h|,
\qquad
\min\{U_h(-1),U_h(1)\}=B_h-|\alpha_h|,
\]
which is equivalent to \eqref{eq:sec5-generic-unstable}. Thus (b) follows. The converse implication is identical.
\end{proof}

Therefore, in the nondegenerate case, we may replace the regular rank-one edge condition of Section~\ref{sec:graph} by the endpoint condition \eqref{eq:sec5-generic-unstable} via Lemma~\ref{lem:sec5-regularity-endpoints}.
Under these conditions, the convex hull of the slice of each unit at $t=\pm 1$ is the three-vertex polytope
\[
\operatorname{conv}\{(L_h(\pm 1),0),(0,0),(U_h(\pm 1),U_h(\pm 1))\}.
\]
In degenerate boundary cases where equality occurs, one of the branches simply disappears, and the argument below becomes even simpler.

\subsection{Description of the one-unit hull by endpoint slices}

For each $h\in\{i,j\}$, define the reduced one-unit graph by
\begin{equation}
\mathcal S_h
:=
\{(t,z_h,x_h)\mid t\in[-1,1],\ L_h(t)\le z_h\le U_h(t),\ x_h=\relu(z_h)\}.
\label{eq:sec5-Sh}
\end{equation}
Also write the convex hulls of the endpoint slices at $t=-1$ and $t=1$ as
\begin{align}
C_h^-&:=\operatorname{conv}\{(L_h(-1),0),(0,0),(U_h(-1),U_h(-1))\},
\label{eq:sec5-Ch-minus}\\
C_h^+&:=\operatorname{conv}\{(L_h(1),0),(0,0),(U_h(1),U_h(1))\}.
\label{eq:sec5-Ch-plus}
\end{align}

Here $C_h^-$ and $C_h^+$ are subsets of the $(z_h,x_h)$-plane, whose first coordinate is the pre-activation $z_h$
and second coordinate is the ReLU output $x_h$. At the endpoints $t=\pm 1$, the one-unit graph consists of the negative branch $(z_h,0)$ and the positive branch $(z_h,z_h)$; thus,
its convex hull is the convex hull of the three points $(L_h(\pm 1),0)$, $(0,0)$, and $(U_h(\pm 1),U_h(\pm 1))$. We use these endpoint slices to describe the exact hulls of both the one-unit and the two-unit problems.

\begin{lemma}[endpoint decomposition of the one-unit hull]
\label{lem:sec5-oneunit-endpoint}
For each $h\in\{i,j\}$,
\begin{equation}
\operatorname{conv}(\mathcal S_h)
=
\operatorname{conv}\bigl(\{-1\}\times C_h^-\ \cup\ \{1\}\times C_h^+\bigr).
\label{eq:sec5-oneunit-endpoint}
\end{equation}
\end{lemma}

\begin{proof}
Write the right-hand side as $H_h$. First, from the definition it is clear that $\{-1\}\times C_h^-\subset \operatorname{conv}(\mathcal S_h)$ and $\{1\}\times C_h^+\subset \operatorname{conv}(\mathcal S_h)$. Hence,
\[
H_h\subset \operatorname{conv}(\mathcal S_h).
\]

Conversely, take any $(t,z_h,x_h)\in\mathcal S_h$ and set
\[
\lambda:=\frac{1-t}{2}\in[0,1].
\]
Then $t=\lambda(-1)+(1-\lambda)(1)$.

If $z_h\le 0$, then $x_h=\relu(z_h)=0$. By the affine dependence in \eqref{eq:sec5-LU},
\[
[L_h(t),0]=\lambda [L_h(-1),0]+(1-\lambda)[L_h(1),0],
\]
thus, there exist $z_h^-\in [L_h(-1),0]$ and $z_h^+\in [L_h(1),0]$ such that
\[
z_h=\lambda z_h^-+(1-\lambda)z_h^+.
\]
Therefore,
\[
(t,z_h,x_h)
=
\lambda(-1,z_h^-,0)+(1-\lambda)(1,z_h^+,0)
\in H_h.
\]

If instead $z_h\ge 0$, then $x_h=z_h$. Similarly,
\[
[0,U_h(t)]=\lambda [0,U_h(-1)]+(1-\lambda)[0,U_h(1)],
\]
hence, there exist $z_h^-\in [0,U_h(-1)]$ and $z_h^+\in [0,U_h(1)]$ such that
\[
z_h=\lambda z_h^-+(1-\lambda)z_h^+.
\]
Hence,
\[
(t,z_h,x_h)
=
\lambda(-1,z_h^-,z_h^-)+(1-\lambda)(1,z_h^+,z_h^+)
\in H_h.
\]

Thus $\mathcal S_h\subset H_h$, and since the right-hand side is convex, we obtain
\[
\operatorname{conv}(\mathcal S_h)\subset H_h.
\]
The proof of \eqref{eq:sec5-oneunit-endpoint} is complete.
\end{proof}

By Lemma~\ref{lem:sec5-oneunit-endpoint}, the one-unit convex hull can be described without treating the
entire
 interval directly: it is the convex combination of the two three-vertex convex hulls at the endpoints $t=\pm 1$. We now lift this fact to the two-unit edge block.

\subsection{Exact decomposition of the two-unit edge hull}

Define the reduced two-unit edge graph by
\begin{equation}
\mathcal S_e
:=
\{(t,z_i,z_j,x_i,x_j)
\mid
(t,z_i,x_i)\in\mathcal S_i,
\ (t,z_j,x_j)\in\mathcal S_j
\}.
\label{eq:sec5-Se}
\end{equation}
Using the convex hull of each unit, define also
\begin{equation}
G_e
:=
\{(t,z_i,z_j,x_i,x_j)
\mid
(t,z_i,x_i)\in \operatorname{conv}(\mathcal S_i),
\ (t,z_j,x_j)\in \operatorname{conv}(\mathcal S_j)
\}.
\label{eq:sec5-Ge}
\end{equation}

\begin{theorem}[exact decomposition of the two-unit edge hull]
\label{thm:sec5-edge-exact}
Under $q_e\le 1$ and \eqref{eq:sec5-generic-unstable},
\begin{equation}
\operatorname{conv}(\mathcal S_e)
=
G_e
=
\operatorname{conv}\bigl(\{-1\}\times C_i^-\times C_j^-\ \cup\ \{1\}\times C_i^+\times C_j^+\bigr).
\label{eq:sec5-edge-exact}
\end{equation}
\end{theorem}

\begin{proof}
First, $\mathcal S_e\subset G_e$ and $G_e$ is convex, so
\[
\operatorname{conv}(\mathcal S_e)\subset G_e.
\]

Conversely, take any
\[
y=(t,z_i,z_j,x_i,x_j)\in G_e.
\]
By Lemma~\ref{lem:sec5-oneunit-endpoint}, with $\lambda=(1-t)/2\in[0,1]$, there exist $a_i^-\in C_i^-$, $a_i^+\in C_i^+$, $a_j^-\in C_j^-$, and $a_j^+\in C_j^+$ such that
\[
(t,z_i,x_i)=\lambda(-1,a_i^-)+(1-\lambda)(1,a_i^+),
\]
\[
(t,z_j,x_j)=\lambda(-1,a_j^-)+(1-\lambda)(1,a_j^+).
\]
Hence,
\[
y
=
\lambda(-1,a_i^-,a_j^-)
+
(1-\lambda)(1,a_i^+,a_j^+),
\]
and therefore,
\[
G_e
\subset
\operatorname{conv}\bigl(\{-1\}\times C_i^-\times C_j^-\cup\{1\}\times C_i^+\times C_j^+\bigr).
\]

On the other hand, both endpoint product slices on the right-hand side are contained in $G_e$.
 By convexity of $G_e$, the reverse inclusion also holds. Thus
\[
G_e
=
\operatorname{conv}\bigl(\{-1\}\times C_i^-\times C_j^-\cup\{1\}\times C_i^+\times C_j^+\bigr).
\]

Finally, when $t=\pm 1$ is fixed, the degrees of freedom on the $i$-side and the $j$-side are independent. Therefore,
\[
\{-1\}\times C_i^-\times C_j^-=
\operatorname{conv}(\mathcal S_e\cap\{t=-1\}),
\]
\[
\{1\}\times C_i^+\times C_j^+=
\operatorname{conv}(\mathcal S_e\cap\{t=1\}),
\]
and hence the right-hand side is contained in $\operatorname{conv}(\mathcal S_e)$. Together with the first inclusion,
it proves \eqref{eq:sec5-edge-exact}.
\end{proof}

Theorem~\ref{thm:sec5-edge-exact} shows that the convex hull of the entire edge block is described exactly as a convex
combination of the two endpoint product slices at the two ends of the shared scalar. Consequently, the required two-unit analysis
is reduced not to a direct description of mixed branches, but to an exact coupling of one-unit interval hulls at the endpoints.
This representation immediately yields a linear extended formulation.

\begin{corollary}[linear extended formulation]
\label{cor:sec5-linear-ef}
The exact hull of Theorem~\ref{thm:sec5-edge-exact} can be represented linearly by using nonnegative variables $\lambda_-,\lambda_+$ and scaled endpoint variables
\[
(\hat z_h^-,\hat x_h^-),\ (\hat z_h^+,\hat x_h^+)
\qquad (h=i,j)
\]
as follows:
\begin{align}
\lambda_-+\lambda_+&=1,
&
 t&=-\lambda_-+\lambda_+,
\label{eq:sec5-ef-lambda}\\
z_h&=\hat z_h^-+\hat z_h^+,
&
 x_h&=\hat x_h^-+\hat x_h^+
\qquad (h=i,j),
\label{eq:sec5-ef-gluing}\\
(\hat z_h^-,\hat x_h^-)&\in \lambda_- C_h^-,
&
(\hat z_h^+,\hat x_h^+)&\in \lambda_+ C_h^+
\qquad (h=i,j).
\label{eq:sec5-ef-triangle}
\end{align}
This is the two-term specialization of the convex-hull formulation for a finite union of polyhedra in disjunctive programming.
Here, $\lambda C_h^{\pm}$ denotes the usual perspective scaling of the polytope $C_h^{\pm}$, which is described by linear inequalities \citep{Balas1979}.
\end{corollary}

\begin{proof}
By Theorem~\ref{thm:sec5-edge-exact}, every point in the exact hull can be written as a convex combination of one point from the endpoint slice at $t=-1$
and one point from the endpoint slice at $t=1$. Thus, there exist $\lambda_-,\lambda_+\ge 0$ with $\lambda_-+\lambda_+=1$, and points
\[
(z_h^-,x_h^-)\in C_h^-,
\qquad
(z_h^+,x_h^+)\in C_h^+
\qquad (h=i,j),
\]
such that
\[
(t,z_i,z_j,x_i,x_j)
=
\lambda_-(-1,z_i^-,z_j^-,x_i^-,x_j^-)
+
\lambda_+(1,z_i^+,z_j^+,x_i^+,x_j^+).
\]
Equating coordinates gives $t=-\lambda_-+\lambda_+$ and, for $h=i,j$,
\[
z_h=\lambda_-z_h^-+\lambda_+z_h^+,
\qquad
x_h=\lambda_-x_h^-+\lambda_+x_h^+.
\]
Introduce the scaled endpoint variables
\[
(\hat z_h^-,\hat x_h^-):=\lambda_-(z_h^-,x_h^-),
\qquad
(\hat z_h^+,\hat x_h^+):=\lambda_+(z_h^+,x_h^+).
\]
Then the preceding coordinate equations become \eqref{eq:sec5-ef-lambda}--\eqref{eq:sec5-ef-gluing},
while the endpoint membership conditions become \eqref{eq:sec5-ef-triangle}.

Conversely, suppose that variables satisfying \eqref{eq:sec5-ef-lambda}--\eqref{eq:sec5-ef-triangle} are given. If $\lambda_\sigma>0$ for $\sigma\in\{-,+\}$, then
\[
(z_h^\sigma,x_h^\sigma):=\lambda_\sigma^{-1}(\hat z_h^\sigma,\hat x_h^\sigma)
\in C_h^\sigma.
\]
If $\lambda_\sigma=0$, the condition $(\hat z_h^\sigma,\hat x_h^\sigma)\in \lambda_\sigma C_h^\sigma$ means $(\hat z_h^\sigma,\hat x_h^\sigma)=(0,0)$, and the corresponding endpoint term has zero weight.
Therefore, any point of $C_h^\sigma$ may be chosen for $(z_h^\sigma,x_h^\sigma)$. Using the equalities in the formulation, the resulting point $(t,z_i,z_j,x_i,x_j)$ belongs to
\[
\operatorname{conv}\left(
\{-1\}\times C_i^-\times C_j^-
\cup
\{1\}\times C_i^+\times C_j^+
\right),
\]
which is $\operatorname{conv}(\mathcal S_e)$ by Theorem~\ref{thm:sec5-edge-exact}. Since each $C_h^\sigma$ is a polytope, the scaled membership condition $(\hat z_h^\sigma,\hat x_h^\sigma)\in\lambda_\sigma C_h^\sigma$ is described by linear inequalities.
Hence, the formulation is linear and exact.
\end{proof}

\subsection{Reduction to an order-\texorpdfstring{$2$}{2} local certificate for the SOS relaxation}

We have now obtained a geometric description of the two-unit edge hull in reduced coordinates.
It remains to convert the polyhedral description into an order-$2$ local certificate.
In this subsection, we show that the conversion can be performed at relaxation order $2$.
The proof has three steps: (i) applying the one-unit result of Azuma et al.~\cite{AKY2026}
 to the reduced one-unit problem, (ii) using a Farkas representation
 for affine optimization over the reduced edge hull, and (iii) lifting the resulting certificate back to the original edge block by affine elimination.

For each $h\in\{i,j\}$, define the reduced one-unit set, with the private scalar 
included explicitly, by
\begin{equation}
\widehat{\mathcal S}_h
:=
\{(t,p_h,z_h,x_h)\mid t\in[-1,1],\ p_h\in[A_h,B_h],\ z_h=p_h+\alpha_h t,\ x_h=\relu(z_h)\}.
\label{eq:sec5-Sh-hat}
\end{equation}
Its projection onto $(t,z_h,x_h)$ is exactly $\mathcal S_h$ in \eqref{eq:sec5-Sh}. For $\sigma\in\{-,+\}$, write
\[
L_h^\sigma:=L_h(\sigma 1),
\qquad
U_h^\sigma:=U_h(\sigma 1).
\]
In the nondegenerate case of Lemma~\ref{lem:sec5-regularity-endpoints}, one has
\[
L_h^\sigma<0<U_h^\sigma
\qquad (h=i,j,\ \sigma\in\{-,+\}).
\]

\begin{lemma}[explicit linear description of the endpoint triangles]
\label{lem:sec5-endpoint-triangle-hrep}
For each $h\in\{i,j\}$ and each $\sigma\in\{-,+\}$, the endpoint triangle $C_h^\sigma$ can be written as
\begin{equation}
C_h^\sigma
=
\left\{(z_h,x_h)\in\R^2\ \middle|\ x_h\ge 0,\ x_h-z_h\ge 0,\ x_h\le \frac{U_h^\sigma}{U_h^\sigma-L_h^\sigma}(z_h-L_h^\sigma)\right\}.
\label{eq:sec5-endpoint-triangle-hrep}
\end{equation}
\end{lemma}

\begin{proof}
Since $C_h^\sigma=\operatorname{conv}\{(L_h^\sigma,0),(0,0),(U_h^\sigma,U_h^\sigma)\}$ is a triangle, it suffices to write the half-spaces corresponding to its three edges. The line through $(L_h^\sigma,0)$ and $(0,0)$ results in
 $x_h=0$, and the line through $(0,0)$ and $(U_h^\sigma,U_h^\sigma)$ leads to $x_h=z_h$.
 The remaining edge through $(L_h^\sigma,0)$ and $(U_h^\sigma,U_h^\sigma)$ has equation
\[
x_h=\frac{U_h^\sigma}{U_h^\sigma-L_h^\sigma}(z_h-L_h^\sigma).
\]
Since the interior of the triangle is the intersection of these half-spaces, we obtain \eqref{eq:sec5-endpoint-triangle-hrep}.
\end{proof}

Thus, at each endpoint $t=\pm 1$, the last inequality in
\eqref{eq:sec5-endpoint-triangle-hrep} is the nontrivial facet of the interval
ReLU hull, as used in convex ReLU formulations
\citep{Ehlers2017,AndersonHuchetteMaTjandraatmadjaVielma2020}.
We use affine inequalities of this type on the one-unit side to represent the optimality gap over the reduced two-unit edge hull.

\begin{lemma}[the reduced one-unit problem falls within the scope of the one-unit result of Azuma et al.~\cite{AKY2026}]
\label{lem:sec5-reduced-azuma}
For each $h\in\{i,j\}$, the set $\widehat{\mathcal S}_h$ in \eqref{eq:sec5-Sh-hat}
corresponds to a single-neuron ReLU problem with input variables $(t,p_h)$ ranging over the rectangle
\[
[-1,1]\times [A_h,B_h].
\]
Therefore, the one-unit tightness result of Azuma et al.~\cite{AKY2026} described in Section~\ref{sec:prelim} applies. In particular, for any affine inequality
\[
\varphi_h(t,z_h,x_h)\ge 0
\]
valid on $\operatorname{conv}(\mathcal S_h)$, we have
\begin{equation}
\varphi_h(t,z_h,x_h)
\in
\qmod_2(g_h^{\mathrm{red}})+\ideal_2(h_h^{\mathrm{red}}),
\label{eq:sec5-reduced-azuma-cert}
\end{equation}
where $(g_h^{\mathrm{red}},h_h^{\mathrm{red}})$ is the family of box constraints and linear equalities
that define the reduced one-unit block. Equivalently,
the one-unit result imported in this section provides order-$2$ certificates without requiring higher-order relaxations.
\end{lemma}

\begin{proof}
The set $\widehat{\mathcal S}_h$ is a one-unit problem with input $(t,p_h)$, pre-activation $z_h$, and output $x_h$, and the input set is the rectangle $[-1,1]\times [A_h,B_h]$. Thus, it is a special case of the rectangular-input one-unit setting of Azuma et al.~\cite{AKY2026}.

Assume that $\varphi_h(t,z_h,x_h)\ge 0$ is valid on $\operatorname{conv}(\mathcal S_h)$. Its pullback to $\widehat{\mathcal S}_h$ is
the affine function
\[
\widehat\varphi_h(t,p_h,z_h,x_h):=\varphi_h(t,z_h,x_h),
\]
which is valid
on $\widehat{\mathcal S}_h$. Let
\[
\widehat m_h:=\min\{\widehat\varphi_h(t,p_h,z_h,x_h)\mid (t,p_h,z_h,x_h)\in \widehat{\mathcal S}_h\}.
\]
Then $\widehat m_h\ge 0$. By the one-unit tightness result of Azuma et al.~\cite{AKY2026} applied to the affine objective $\widehat\varphi_h$,
we have 
\[
\widehat\varphi_h-\widehat m_h
\in
\qmod_2(\widehat g_h)+\ideal_2(\widehat h_h).
\]
Since $\widehat m_h$ is a nonnegative constant, it belongs to the quadratic module as an SOS polynomial. Therefore,
\[
\widehat\varphi_h
\in
\qmod_2(\widehat g_h)+\ideal_2(\widehat h_h).
\]
Finally, applying Lemma~\ref{lem:sec5-affine-elimination}(i) to the affine substitution $p_h=z_h-\alpha_h t$ transfers this certificate
 to \eqref{eq:sec5-reduced-azuma-cert} for the reduced block $(g_h^{\mathrm{red}},h_h^{\mathrm{red}})$.
\end{proof}

Pullbacks also play a role in positivity certificates for polynomial optimization
\citep{LasserrePutinar2010}.
Here, however, the required maps are affine, and therefore preserve SOS structure without increasing the degree.

\begin{lemma}[affine inequalities valid on a box have order-$2$ box certificates]
\label{lem:sec5-box-affine-certificate}
Let
\[
B=[\underline u_1,\bar u_1]\times\cdots\times[\underline u_m,\bar u_m]
\]
and let
\[
g_B:=\{u_k-\underline u_k,\ \bar u_k-u_k\mid k=1,\dots,m\}
\]
be its box generators. If an affine polynomial
\[
q(u)=q_0+\sum_{k=1}^m a_k u_k
\]
is nonnegative on $B$, then
\[
q\in \qmod_2(g_B).
\]
\end{lemma}

\begin{proof}
The minimum of $q$ on the box is attained at the vertex
\[
u_k^\star=
\begin{cases}
\underline u_k, & a_k\ge 0,\\
\bar u_k, & a_k<0.
\end{cases}
\]
Set $m_q:=q(u^\star)\ge 0$. Then
\[
q(u)
=
m_q
+
\sum_{a_k\ge 0} a_k\,(u_k-\underline u_k)
+
\sum_{a_k<0} (-a_k)\,(\bar u_k-u_k).
\]
Since $m_q$ is a nonnegative constant, it is an SOS polynomial. Moreover,
all coefficients multiplying the box generators are nonnegative constants.
Hence, this representation belongs to $\qmod_2(g_B)$.
\end{proof}

\begin{lemma}[affine pullback preserves order-$2$ certificates]
\label{lem:sec5-affine-elimination}
The following statements hold.
\begin{enumerate}[label=(\roman*)]
\item For each $h\in\{i,j\}$, consider the affine substitution
\[
\Phi_h:\R[t,p_h,z_h,x_h]\to \R[t,z_h,x_h],
\qquad
(\Phi_h q)(t,z_h,x_h):=q(t,z_h-\alpha_h t,z_h,x_h).
\]
If $\widehat{\psi}(t,p_h,z_h,x_h)\in \qmod_2(\widehat g_h)+\ideal_2(\widehat h_h)$, then
\[
\psi:=\Phi_h(\widehat\psi)
\in
\qmod_2(g_h^{\mathrm{red}})+\ideal_2(h_h^{\mathrm{red}})
\]
holds.
\item Let $\pi_e$ be the affine map from the original edge-local variables $\zeta_e$ to the reduced variables $\xi=(t,z_i,x_i,z_j,x_j)$, and let its pullback be
\[
\Pi_e:\R[\xi]\to \R[\zeta_e],
\qquad
(\Pi_e q)(\zeta_e):=q(\pi_e(\zeta_e)).
\]
If a reduced polynomial $\widetilde\psi(\xi)$ satisfies
\[
\widetilde\psi\in \qmod_2(g_e^{\mathrm{red}})+\ideal_2(h_e^{\mathrm{red}}),
\]
then its pullback satisfies
\[
\Pi_e(\widetilde\psi)
=
\widetilde\psi(\pi_e(\zeta_e))
\in
\qmod_2(g_e)+\ideal_2(h_e).
\]
\end{enumerate}
Moreover, in both (i) and (ii), the polynomial degree is preserved under the pullback. Hence,
the degree conditions associated with relaxation order $2$ remain valid.
\end{lemma}

\begin{proof}
(i) An affine substitution is a ring homomorphism. Hence if $\widehat\sigma=\sum_m r_m^2$ is SOS, then
\[
\Phi_h(\widehat\sigma)=\sum_m \bigl(r_m(t,z_h-\alpha_h t,z_h,x_h)\bigr)^2
\]
is again SOS. Since the map $p_h\mapsto z_h-\alpha_h t$ is affine,
it does not increase degree. Hence,
\[
\deg \Phi_h(q)\le \deg q
\]
for every polynomial $q$.

Now, write
\[
\widehat\psi
=
\widehat\sigma_0+
\sum_\ell \widehat\sigma_\ell \,\widehat g_{h,\ell}
+
\sum_r \widehat\tau_r \,\widehat h_{h,r}
\in
\qmod_2(\widehat g_h)+\ideal_2(\widehat h_h).
\]
Each $\widehat\sigma_\ell$ is SOS, and each term satisfies the degree constraints associated with relaxation order $2$.
After substituting $p_h=z_h-\alpha_h t$, we obtain
\[
\Phi_h(p_h-A_h)=z_h-L_h(t),
\qquad
\Phi_h(B_h-p_h)=U_h(t)-z_h.
\]
Thus, the box constraints associated with $t\in[-1,1]$ are preserved. Moreover, the linear equality
\[
z_h-p_h-\alpha_h t=0
\]
becomes identically zero after pullback. Therefore, $\Phi_h(\widehat\psi)$ belongs to the sum of the truncated quadratic module and ideal
associated with the reduced one-unit block $(g_h^{\mathrm{red}},h_h^{\mathrm{red}})$. Since no term increases in degree, we obtain
\[
\Phi_h(\widehat\psi)
\in
\qmod_2(g_h^{\mathrm{red}})+\ideal_2(h_h^{\mathrm{red}}).
\]

	(ii) Similarly, $\Pi_e$ is the ring homomorphism induced by the affine map $\pi_e$,
	and therefore sends SOS polynomials to SOS polynomials without increasing degree.
	It remains only to track the reduced generators under the pullback.  
	The ReLU generators $x_h$ and $x_h-z_h$ become the original generators $x_h$ and $x_h-v_h$, respectively,
	while the reduced complementarity equality $x_h(x_h-z_h)=0$ becomes the original equality $x_h(x_h-v_h)=0$.

	The remaining reduced generators are affine box-type bounds. The pullbacks of $1+t$ and $1-t$ are affine functions
	that are nonnegative on the shared input box by the normalization of $t$.
	Similarly, the pullbacks of $z_h-L_h(t)$ and $U_h(t)-z_h$ are $p_h-A_h$ and $B_h-p_h$, respectively,
	and these affine functions are nonnegative on the private input box by the definitions of $A_h$ and $B_h$.
	By Lemma~\ref{lem:sec5-box-affine-certificate}, each of these affine pullbacks belongs to the order-$2$
	quadratic module generated by the corresponding original box constraints. Therefore, every reduced
	inequality generator pulls back either to an original ReLU generator or to an element of the original box quadratic
	module, while every reduced equality generator pulls back to the original ideal.

	Thus if
	\[
	\widetilde\psi
	\in
	\qmod_2(g_e^{\mathrm{red}})+\ideal_2(h_e^{\mathrm{red}}),
	\]
	then applying $\Pi_e$ term by term maps it into $\qmod_2(g_e)+\ideal_2(h_e)$. The degree bounds are preserved
	since the pullback is affine and Lemma~\ref{lem:sec5-box-affine-certificate} represents each pulled-back affine box-type
	generator as a nonnegative constant combination of affine original box generators. Consequently,
	\[
	\Pi_e(\widetilde\psi)
	\in
\qmod_2(g_e)+\ideal_2(h_e).
\]
\end{proof}

\begin{lemma}[adding unused variables preserves order-$2$ certificates]
\label{lem:sec5-variable-extension}
For each $h\in\{i,j\}$, suppose
\[
\psi_h(t,z_h,x_h)
\in
\qmod_2(g_h^{\mathrm{red}})+\ideal_2(h_h^{\mathrm{red}}).
\]
Then, viewing $\psi_h$ as a polynomial in the reduced edge variables $(t,z_i,x_i,z_j,x_j)$, we also have
\[
\psi_h
\in
\qmod_2(g_e^{\mathrm{red}})+\ideal_2(h_e^{\mathrm{red}}).
\]
\end{lemma}

\begin{proof}
Each generator of the reduced one-unit block appears verbatim as part of the generators of the reduced edge block. Therefore, any representation
\[
\psi_h
=
\sigma_0+
\sum_\ell \sigma_\ell g_{h,\ell}^{\mathrm{red}}
+
\sum_r \tau_r h_{h,r}^{\mathrm{red}}
\]
in $\qmod_2(g_h^{\mathrm{red}})+\ideal_2(h_h^{\mathrm{red}})$ remains a valid representation in
$\qmod_2(g_e^{\mathrm{red}})+\ideal_2(h_e^{\mathrm{red}})$,
after viewing the coefficient polynomials as independent of the unused variables $(z_{h'},x_{h'})$.
The SOS structure and the degree bounds are unchanged, which proves the claim.
\end{proof}

\begin{lemma}[affine optimization over the reduced edge graph reduces to linear programming]
\label{lem:sec5-edge-lp}
Assume the nondegenerate case $\underline\tau_e<\overline\tau_e$, and let the reduced edge hull be
\[
P_e:=\operatorname{conv}(\mathcal S_e).
\]
Then, for every reduced affine objective $\widetilde\ell_e$,
\begin{equation}
\min\{\widetilde\ell_e(\xi)\mid \xi\in \mathcal S_e\}
=
\min\{\widetilde\ell_e(\xi)\mid \xi\in P_e\}
\label{eq:sec5-edge-lp-hull}
\end{equation}
holds. Moreover, using the linear extended formulation of Corollary~\ref{cor:sec5-linear-ef}, the right-hand side can be solved as a linear program in lifted variables.
\end{lemma}

\begin{proof}
Since $\widetilde\ell_e$ is affine, its value at a convex combination of points in $\mathcal S_e$ is the same convex combination of the corresponding values. Hence, no point in $\operatorname{conv}(\mathcal S_e)$ can have a smaller value than the minimum over $\mathcal S_e$, and the reverse inequality is immediate from $\mathcal S_e\subset\operatorname{conv}(\mathcal S_e)$. Therefore,
\[
\min\{\widetilde\ell_e(\xi)\mid \xi\in \mathcal S_e\}
=
\min\{\widetilde\ell_e(\xi)\mid \xi\in \operatorname{conv}(\mathcal S_e)\}.
\]
Since $P_e=\operatorname{conv}(\mathcal S_e)$, this proves \eqref{eq:sec5-edge-lp-hull}. The second claim follows from Corollary~\ref{cor:sec5-linear-ef}, which provides a linear extended formulation of $P_e$.
\end{proof}

\begin{lemma}[Farkas representation of the optimality gap on the reduced edge hull]
\label{lem:sec5-edge-farkas}
Assume the nondegenerate case $\underline\tau_e<\overline\tau_e$. For each $h\in\{i,j\}$, let
\[
P_h:=\operatorname{conv}(\mathcal S_h).
\]
Then, for any reduced affine objective $\widetilde\ell_e$, if its reduced optimal value is
\[
\widetilde\gamma_e^\star:=\min\{\widetilde\ell_e(\xi)\mid \xi\in P_e\},
\]
there exist finitely many affine inequalities $\varphi_{i,1},\dots,\varphi_{i,N_i}$ valid on $P_i$, affine inequalities $\varphi_{j,1},\dots,\varphi_{j,N_j}$ valid on $P_j$, and nonnegative coefficients $\lambda_{i,1},\dots,\lambda_{i,N_i},\lambda_{j,1},\dots,\lambda_{j,N_j}$ such that
\begin{equation}
\widetilde\ell_e-\widetilde\gamma_e^\star
=
\sum_{r=1}^{N_i}\lambda_{i,r}\,\varphi_{i,r}(t,z_i,x_i)
+
\sum_{s=1}^{N_j}\lambda_{j,s}\,\varphi_{j,s}(t,z_j,x_j)
\label{eq:sec5-edge-farkas}
\end{equation}
holds on the reduced edge variables.
\end{lemma}

\begin{proof}
By Theorem~\ref{thm:sec5-edge-exact}, the reduced edge hull can be written as
\[
P_e
=
\{(t,z_i,z_j,x_i,x_j)\mid (t,z_i,x_i)\in P_i,\ (t,z_j,x_j)\in P_j\}.
\]
Moreover, by Lemma~\ref{lem:sec5-oneunit-endpoint} and Corollary~\ref{cor:sec5-linear-ef}, each $P_h$ is a polytope. Hence, we can write
\[
P_h=\{(t,z_h,x_h)\mid \varphi_{h,r}(t,z_h,x_h)\ge 0\ \text{for }r=1,\dots,N_h\}
\]
using finitely many valid affine inequalities.

Let $M_i(t,z_i,x_i,z_j,x_j)=(t,z_i,x_i)$ and $M_j(t,z_i,x_i,z_j,x_j)=(t,z_j,x_j)$. Then $P_e$ admits the polyhedral description
\[
P_e=
\{\xi\mid \varphi_{i,r}(M_i\xi)\ge 0\ (r=1,\dots,N_i),\ \varphi_{j,s}(M_j\xi)\ge 0\ (s=1,\dots,N_j)\}.
\]
Each row is simply an affine inequality valid on $P_i$ or $P_j$, embedded into the reduced edge variables.

Now, $\widetilde\gamma_e^\star$ is the minimum of the affine function $\widetilde\ell_e$ over the polytope $P_e$. Therefore, by the polyhedral form of Farkas' lemma, there exist nonnegative vectors $\lambda_i,\lambda_j$ such that
\[
\widetilde\ell_e-\widetilde\gamma_e^\star
=
\sum_{r=1}^{N_i}\lambda_{i,r}\,\varphi_{i,r}(M_i\xi)
+
\sum_{s=1}^{N_j}\lambda_{j,s}\,\varphi_{j,s}(M_j\xi).
\]
Renaming the rows as $\varphi_{i,r}$ and $\varphi_{j,s}$ yields \eqref{eq:sec5-edge-farkas}.
\end{proof}

Before applying the reduced edge hull description to the original edge block, we specify
the class of affine objectives covered by the reduced-hull argument.
Let $\ell_e$ be an affine edge-local objective whose dependence on the original input variables is given by
\[
d_i^\top u_{T_i}+d_e^\top u_{T_e}+d_j^\top u_{T_j},
\]
where $T_i$, $T_e$, and $T_j$ denote, respectively, the private input set of unit $i$, the shared input set, and the private input set of unit $j$. We say that $\ell_e$ is \emph{expressible in the rank-one reduced variables} if
\begin{equation}
d_i\in \operatorname{span}\{\bar W_{i,T_i}^{\top}\},
\qquad
d_j\in \operatorname{span}\{\bar W_{j,T_j}^{\top}\},
\qquad
d_e\in \operatorname{span}\{r_e\}.
\label{eq:sec5-objective-compatible}
\end{equation}
We call \eqref{eq:sec5-objective-compatible} the reduced-objective condition. Under this condition,
the input-variable dependence of $\ell_e$ is affine in the reduced scalars $p_i$, $p_j$, and $t_0=r_e^\top u_{T_e}$.
After normalizing $t_0$ and using the substitution $p_h=z_h-\alpha_h t$, the objective restricted to the feasible edge block
can be expressed as a reduced affine objective, 
\[
\ell_e(\zeta_e)=\widetilde\ell_e(t,z_i,x_i,z_j,x_j).
\]

\begin{theorem}[edge-local tightness for a regular rank-one edge at relaxation order $2$]
\label{thm:sec5-edge-degree2}
Assume that the edge $e=\{i,j\}$ is a regular rank-one edge in the sense of Section~\ref{sec:graph}. Let $\ell_e$ be an affine objective on the edge block satisfying the reduced-objective condition \eqref{eq:sec5-objective-compatible}, and let the local optimal value be denoted by
\[
\gamma_e^\star:=\min\{\ell_e(\zeta_e)\mid \zeta_e\in K_e\}.
\]
Then,
\begin{equation}
\ell_e-\gamma_e^\star
\in
\qmod_2(g_e)+\ideal_2(h_e)
\label{eq:sec5-edge-degree2}
\end{equation}
holds.
\end{theorem}

\begin{proof}
We first consider
the degenerate case $\underline\tau_e=\overline\tau_e$. By Lemma~\ref{lem:sec5-regularity-endpoints},
the edge block then reduces to the direct product of two one-unit blocks. Applying the one-unit tightness result of Azuma et al.~\cite{AKY2026}
 to each block and summing the resulting certificates yields
an order-$2$ certificate for this direct-product block. Pulling this certificate back through
the affine elimination in Lemma~\ref{lem:sec5-affine-elimination}(ii), we obtain
\[
\ell_e-\gamma_e^\star\in \qmod_2(g_e)+\ideal_2(h_e).
\]
We now turn to the nondegenerate case $\underline\tau_e<\overline\tau_e$.

Using the affine elimination introduced at the beginning of this section, let $\pi_e$ map the original edge variables $\zeta_e$ to
\[
\xi:=(t,z_i,x_i,z_j,x_j).
\]
By the reduced-objective condition, there exists a reduced affine objective $\widetilde\ell_e$ such that, on the feasible edge block,
\[
\ell_e(\zeta_e)=\widetilde\ell_e(\pi_e(\zeta_e)).
\]
Set 
\[
\widetilde\gamma_e^\star:=\min\{\widetilde\ell_e(\xi)\mid \xi\in P_e\}.
\]
Moreover, the image of the original feasible edge block under $\pi_e$ is the reduced graph $\mathcal S_e$,
and $P_e=\operatorname{conv}(\mathcal S_e)$. Since $\widetilde\ell_e$ is affine,
its minimum over $\mathcal S_e$ coincides with its minimum over $P_e$. Hence,
\[
\widetilde\gamma_e^\star=\gamma_e^\star.
\]

Next, apply Lemma~\ref{lem:sec5-edge-farkas} to $\widetilde\ell_e$. Then, there exist affine inequalities $\varphi_{i,r}$ valid on $P_i$,
affine inequalities $\varphi_{j,s}$ valid on $P_j$, and nonnegative coefficients $\lambda_{i,r},\lambda_{j,s}$ such that
\[
\widetilde\ell_e-\widetilde\gamma_e^\star
=
\sum_r \lambda_{i,r}\,\varphi_{i,r}(t,z_i,x_i)
+
\sum_s \lambda_{j,s}\,\varphi_{j,s}(t,z_j,x_j).
\]
Each $\varphi_{i,r}$ and $\varphi_{j,s}$ is valid on the corresponding reduced one-unit hull.
Thus, by Lemma~\ref{lem:sec5-reduced-azuma}, we have
\[
\varphi_{i,r}\in \qmod_2(g_i^{\mathrm{red}})+\ideal_2(h_i^{\mathrm{red}}),
\qquad
\varphi_{j,s}\in \qmod_2(g_j^{\mathrm{red}})+\ideal_2(h_j^{\mathrm{red}}).
\]
Applying Lemma~\ref{lem:sec5-variable-extension} to each term, we may regard these terms as 
elements of
\[
\qmod_2(g_e^{\mathrm{red}})+\ideal_2(h_e^{\mathrm{red}}).
\]
Since the coefficients are nonnegative, their sum belongs to the same truncated quadratic module plus ideal. Therefore,
\[
\widetilde\ell_e-\widetilde\gamma_e^\star
\in
\qmod_2(g_e^{\mathrm{red}})+\ideal_2(h_e^{\mathrm{red}}).
\]
Finally, applying Lemma~\ref{lem:sec5-affine-elimination}(ii) to $\widetilde\psi=\widetilde\ell_e-\widetilde\gamma_e^\star$ gives a pullback
certificate satisfying
\[
\ell_e-\gamma_e^\star
\in
\qmod_2(g_e)+\ideal_2(h_e),
\]
which is exactly \eqref{eq:sec5-edge-degree2}.
\end{proof}

Thus, a regular rank-one edge is handled by transferring the one-unit result of Azuma et al.~\cite{AKY2026}
through the reduced hull and the affine pullback.

\subsection{Main theorem for matching}

We use the notation of Section~\ref{sec:reduction} and write 
\[
\mathcal I=\{j\in\Uset\mid \deg_H(j)=0\}
\]
for the set of isolated vertices.
The final global theorem is proved under the following three structural assumptions.
\begin{enumerate}[label=\textbf{(M\arabic*)},leftmargin=24mm]
\item Each isolated vertex block $(g_j,h_j)$, $j\in\mathcal I$, satisfies at least one of the one-unit conditions of Azuma et al.~\cite{AKY2026}.
\item The input-sharing graph $H=(\Uset,E)$ is a matching.
\item Every edge $e\in E$ is a regular rank-one edge as defined in Section~\ref{sec:graph}.
\end{enumerate}

\begin{theorem}[main theorem for matching]
\label{thm:sec5-matching-main}
Assume that \textbf{(M1)}--\textbf{(M3)} hold. Suppose further that, for the given output objective,
every edge-local affine objective $\Theta_e^{\mathrm{lin}}$ defined in \eqref{eq:sec4-component-edge}
satisfies the reduced-objective condition \eqref{eq:sec5-objective-compatible}.
Then, the order-$2$ block-sparse SOS relaxation induced by the connected components of the matching is tight.
In particular,
\begin{equation}
\gamma_{\mathrm{sp}}^{(2)}=\gamma^*.
\label{eq:sec5-matching-main}
\end{equation}
\end{theorem}

\begin{proof}
First, let $j\in\mathcal I$ be any isolated vertex. By assumption \textbf{(M1)}, the one-unit tightness result of Azuma et al.~\cite{AKY2026}
from Section~\ref{sec:prelim} applies to the local problem $(g_j,h_j)$. Therefore, for the local affine objective $\Theta_j^{\mathrm{lin}}$ defined in Section~\ref{sec:reduction}, we have
\[
\Theta_j^{\mathrm{lin}}-m_j
\in
\qmod_2(g_j)+\ideal_2(h_j),
\]
which is exactly \eqref{eq:sec4-local-cert-vertex}.

Next, for any edge $e\in E$, assumption \textbf{(M3)} implies that $e$ is a regular rank-one edge.
The reduced-objective condition in the theorem statement ensures that
Theorem~\ref{thm:sec5-edge-degree2} applies to the local affine objective $\ell_e=\Theta_e^{\mathrm{lin}}$. Therefore,
\[
\Theta_e^{\mathrm{lin}}-m_e
\in
\qmod_2(g_e)+\ideal_2(h_e),
\]
which is \eqref{eq:sec4-local-cert-edge}.

Finally, by assumption \textbf{(M2)}, the graph $H$ is a matching. Hence Proposition~\ref{prop:sec4-global-exactness}
 from Section~\ref{sec:reduction} applies to the component decomposition.
 Summing the local certificates established above for the isolated vertices and single-edge components
 yields a global certificate. Therefore
\[
\gamma_{\mathrm{sp}}^{(2)}=\gamma^*.
\]
\end{proof}

\begin{remark}
The reduced-objective condition \eqref{eq:sec5-objective-compatible} is objective-dependent and is not an additional graph assumption.
With \textbf{(M1)}--\textbf{(M3)}, it is directly checkable from the reduced data and the input box. Extending Theorem~\ref{thm:sec5-matching-main}
 to objectives outside this class would require retaining the input directions eliminated by the rank-one reduction.
\end{remark}

%% file: section-06.tex
\section{Conclusion and Future Work}\label{sec:conclusion}

We have established tightness at relaxation order $2$ for a block-sparse SOS relaxation of one-layer ReLU verification under the matching, regular rank-one, and reduced-objective conditions. The matching structure decomposes the problem into isolated-vertex and single-edge blocks. Existing one-unit results handle the former, while the endpoint-slice description developed here yields local certificates for the latter. Summing these certificates provides the global tightness result.

Extending the analysis to general forests remains open since vertices of degree at least $2$ introduce compatibility conditions that must be represented by low-degree sparse SOS certificates. Other directions include edges with shared rank $q_e\ge 2$, graphs with cycles, and higher-degree certificates for structures beyond rank-one regularity.

%% file: section-0a.tex
\section{A Concrete Example Illustrating the Theory of Sections 3--5}\label{app:worked-example}

In this appendix, we present a small concrete instance that illustrates the framework developed in Sections~\ref{sec:graph}--\ref{sec:edge}. We first show how the input-sharing graph, the matching structure, and the regular rank-one edge condition are obtained from the input data $(W,b,\mathcal U)$. We then compute the local optimal values $m_e$ and $m_3$ through the component decomposition and derive the optimal value $\gamma^*=13/2$ of the original ReLU verification problem. Finally, we give a reduced certificate of relaxation order $2$ for the edge block and confirm tightness of the block-sparse SOS relaxation for this instance.

Appendix~A.1 describes the original verification problem and the reduced data. Appendix~A.2 constructs the input-sharing graph and verifies that its unique edge is a regular rank-one edge. Appendix~A.3 computes the optimal value $\gamma^*=13/2$ from the component decomposition. Appendix~A.4 introduces the reduced polynomial optimization problem for the edge block and presents a relaxation-order-$2$ certificate.

\subsection{The original ReLU verification problem and the reduced data}\label{app:subsec:data}

Let the input dimension be $n_0=4$, the number of hidden units be $n_1=4$, and the output dimension be $n_2=1$. Let the input set be
\begin{equation}
\mathcal U=[-1,1]^4.
\label{eq:app-U}
\end{equation}
We define the weight matrix and bias by
\begin{equation}
W=
\begin{pmatrix}
\frac12 & 1 & 0 & 0\\
-\frac12 & 0 & 1 & 0\\
0 & 0 & 0 & 1\\
0 & 1 & 0 & 0
\end{pmatrix},
\qquad
b=
\begin{pmatrix}
0\\0\\0\\2
\end{pmatrix},
\label{eq:app-Wb}
\end{equation}
and the output map by
\begin{equation}
A=
\begin{pmatrix}
1 & 2 & -1 & 1
\end{pmatrix},
\qquad a=0.
\label{eq:app-Aa}
\end{equation}
Let the linear specification vector be $c=1$. Then the original ReLU verification problem is
\begin{equation}
\gamma^*=
\max\{Ax \mid u\in\mathcal U,\ v=Wu+b,\ x=\relu(v)\}.
\label{eq:app-verification}
\end{equation}

The pre-activations are given by
\begin{align}
v_1&=\tfrac12 u_1+u_2,
&
v_2&=-\tfrac12 u_1+u_3,
&
v_3&=u_4,
&
v_4&=u_2+2.
\label{eq:app-vs}
\end{align}
Therefore, from \eqref{eq:app-U} one immediately obtains
\begin{align}
v_1&\in[-\tfrac32,\tfrac32],
&
v_2&\in[-\tfrac32,\tfrac32],
&
v_3&\in[-1,1],
&
v_4&\in[1,3].
\label{eq:app-v-intervals}
\end{align}
These interval bounds show that units $1,2,3$ are undecided, while unit $4$ is a determined positive unit. Hence,
\begin{equation}
\Uset=\{1,2,3\},
\qquad
\mathcal J^+=\{4\},
\qquad
\mathcal J^- = \emptyset.
\label{eq:app-partition}
\end{equation}

Following the reduced notation of Section~\ref{sec:graph}, the data corresponding to the undecided units are
\begin{equation}
\bar W=
\begin{pmatrix}
\frac12 & 1 & 0 & 0\\
-\frac12 & 0 & 1 & 0\\
0 & 0 & 0 & 1
\end{pmatrix},
\qquad
\bar b=
\begin{pmatrix}
0\\0\\0
\end{pmatrix},
\qquad
\bar A=
\begin{pmatrix}
1 & 2 & -1
\end{pmatrix}.
\label{eq:app-reduced-data}
\end{equation}
On the other hand, the contribution of the determined positive unit $4$ is
\begin{equation}
A_{:,\mathcal J^+}W_{\mathcal J^+,:}=
\begin{pmatrix}0 & 1 & 0 & 0\end{pmatrix},
\qquad
A_{:,\mathcal J^+}b_{\mathcal J^+}=2,
\label{eq:app-pos-unit-contrib}
\end{equation}
and after incorporating this into the objective, the reduced verification problem becomes
\begin{equation}
\gamma^*=
\max\{x_1+2x_2-x_3+u_2+2 \mid u\in[-1,1]^4,\ v=\bar W u,\ x=\relu(v)\}.
\label{eq:app-reduced-verification}
\end{equation}

\subsection{Input-sharing graph, matching structure, and regular rank-one edge}\label{app:subsec:graph}

The input supports of the undecided units are
\begin{equation}
S_1=\{1,2\},
\qquad
S_2=\{1,3\},
\qquad
S_3=\{4\}.
\label{eq:app-supports}
\end{equation}
Hence the edge set of the input-sharing graph $H=(\Uset,E)$ is
\begin{equation}
E=\bigl\{\{1,2\}\bigr\}.
\label{eq:app-edge-set}
\end{equation}
That is, units $1$ and $2$ share the input coordinate $u_1$, whereas unit $3$ is an isolated vertex. Thus the input-sharing graph is a matching, and the global decomposition in the paper appears here in its simplest nontrivial form.

For the edge $e=\{1,2\}$, the shared input set and private sets are
\begin{equation}
S_e=\{1\},
\qquad
P_1=\{2\},
\qquad
P_2=\{3\}.
\label{eq:app-shared-private}
\end{equation}
The shared rank is
\begin{equation}
q_e=
\dim\operatorname{span}\left\{\begin{pmatrix}\frac12\end{pmatrix},\begin{pmatrix}-\frac12\end{pmatrix}\right\}=1,
\label{eq:app-qe}
\end{equation}
so condition \eqref{eq:M3a-sec3} holds.

Furthermore,
\begin{align}
v_1&=u_2+\tfrac12u_1,
&
v_2&=u_3-\tfrac12u_1,
\end{align}
so the lower and upper bounds of the private and shared parts are
\begin{align}
\underline p_{1,e}=-1,
&\qquad \overline p_{1,e}=1,
&
\underline s_{1,e}=-\tfrac12,
&\qquad \overline s_{1,e}=\tfrac12,\\
\underline p_{2,e}=-1,
&\qquad \overline p_{2,e}=1,
&
\underline s_{2,e}=-\tfrac12,
&\qquad \overline s_{2,e}=\tfrac12.
\label{eq:app-ps-bounds}
\end{align}
Therefore the edge-regularity condition \eqref{eq:M3b-sec3} becomes
\begin{equation}
-1+\tfrac12<0<1-\tfrac12,
\label{eq:app-edge-regularity-check}
\end{equation}
which holds for both units. Thus the unique edge $e=\{1,2\}$ is a regular rank-one edge.

\subsection{Component decomposition and the optimal value of the original ReLU problem}\label{app:subsec:decomp}

In this example, the set of isolated vertices is
\begin{equation}
\mathcal I=\{3\}.
\label{eq:app-isolated-set}
\end{equation}
Using the notation of Section~\ref{sec:reduction}, we then have
\begin{equation}
T_1=\{2\},
\qquad
T_e=\{1\},
\qquad
T_2=\{3\},
\qquad
T_3=\{4\},
\qquad
T_0=\emptyset.
\label{eq:app-Tsets}
\end{equation}
Moreover, the coefficients defined in \eqref{eq:sec4-alpha-delta} are
\begin{equation}
\alpha_1=1,
\qquad
\alpha_2=2,
\qquad
\alpha_3=-1,
\qquad
(\delta_1,\delta_2,\delta_3,\delta_4)=(0,-1,0,0).
\label{eq:app-alpha-delta}
\end{equation}
Hence the local affine objectives are given by
\begin{equation}
\Theta_e^{\mathrm{lin}}=-x_1-2x_2-u_2,
\label{eq:app-edge-objective}
\end{equation}
and
\begin{equation}
\Theta_3^{\mathrm{lin}}=x_3.
\label{eq:app-vertex-objective}
\end{equation}
Therefore the gap polynomial decomposes as
\begin{equation}
p_\gamma=\gamma-2+\Theta_e^{\mathrm{lin}}+\Theta_3^{\mathrm{lin}}.
\label{eq:app-pgamma-decomp}
\end{equation}

On the isolated-vertex side, the local optimal value is immediately
\begin{equation}
m_3=0.
\label{eq:app-m3}
\end{equation}
For the edge component, introduce the reduced coordinates
\begin{equation}
t:=u_1,
\qquad
p_1:=u_2,
\qquad
p_2:=u_3,
\qquad
z_1=p_1+\tfrac12 t,
\qquad
z_2=p_2-\tfrac12 t.
\label{eq:app-reduced-coordinates}
\end{equation}
The edge-local objective is expressible in the rank-one reduced variables introduced in Section~\ref{sec:edge}: its only input-linear term is $-u_2=-p_1$, and it contains no input-linear component orthogonal to the reduced shared scalar $t=u_1$.
Then the local objective becomes
\begin{equation}
\Theta_e^{\mathrm{lin}}=-p_1-x_1-2x_2.
\label{eq:app-edge-objective-reduced}
\end{equation}
For fixed $t$, minimizing $\Theta_e^{\mathrm{lin}}$ is equivalent to maximizing $p_1+x_1+2x_2$. Since $x_1=\relu(p_1+\tfrac12 t)$ is nondecreasing in $p_1$ and $x_2=\relu(p_2-\tfrac12 t)$ is nondecreasing in $p_2$, the optimum is attained at $p_1=p_2=1$. In that case,
\begin{equation}
z_1=1+\tfrac12 t,
\qquad
z_2=1-\tfrac12 t,
\label{eq:app-z-at-opt}
\end{equation}
so the quantity to be maximized reduces to
\begin{equation}
1+\tfrac12 t+2\left(1-\tfrac12 t\right)+1
=4-\tfrac12 t.
\label{eq:app-reduced-max}
\end{equation}
Therefore the maximum value is $9/2$ at $t=-1$, and hence
\begin{equation}
m_e=-\frac92.
\label{eq:app-me}
\end{equation}

Consequently,
\begin{equation}
\eta_{\gamma^*}+m_e+m_3=0,
\label{eq:app-balance}
\end{equation}
so the optimal value of the reduced problem is
\begin{equation}
\gamma^*=\frac{13}{2}.
\label{eq:app-gammastar}
\end{equation}
Moreover, in the original ReLU problem \eqref{eq:app-verification}, if we take
\begin{equation}
u^*=(-1,1,1,-1),
\label{eq:app-optu}
\end{equation}
then
\begin{equation}
v=(\tfrac12,\tfrac32,-1,3),
\qquad
x=(\tfrac12,\tfrac32,0,3),
\label{eq:app-optvx}
\end{equation}
and
\begin{equation}
Ax=\frac12+2\cdot\frac32+3=\frac{13}{2}.
\label{eq:app-optvalue-check}
\end{equation}
Thus \eqref{eq:app-optu} gives an explicit optimal solution of the original verification problem.

\subsection{A reduced-edge certificate of relaxation order \texorpdfstring{$2$}{2}}\label{app:subsec:certificate}

We next give an order-$2$ local certificate for the unique edge block $e=\{1,2\}$. In this example, let the variables of the reduced edge block be
\begin{equation}
(t,z_1,x_1,z_2,x_2),
\label{eq:app-reduced-edge-vars}
\end{equation}
and define the target polynomial by
\begin{equation}
\widetilde p_e(t,z_1,x_1,z_2,x_2):=\frac92+\frac12 t-z_1-x_1-2x_2.
\label{eq:app-reduced-pe}
\end{equation}
This is obtained from \eqref{eq:app-edge-objective-reduced} by substituting $p_1=z_1-\tfrac12 t$, and it is precisely the gap polynomial corresponding to the local optimum $m_e=-9/2$.

For this reduced block, the box-type generators are
\begin{align}
g_1&=1+t,
&g_2&=1-t,
\nonumber\\
g_3&=z_1+1-\tfrac12 t,
&g_4&=1+\tfrac12 t-z_1,
\nonumber\\
g_5&=x_1,
&g_6&=x_1-z_1,
\nonumber\\
g_7&=z_2+1+\tfrac12 t,
&g_8&=1-\tfrac12 t-z_2,
\nonumber\\
g_9&=x_2,
&g_{10}&=x_2-z_2,
\label{eq:app-reduced-generators}
\end{align}
and the ideal generators arising from complementarity are
\begin{equation}
h_1=x_1(x_1-z_1),
\qquad
h_2=x_2(x_2-z_2).
\label{eq:app-reduced-ideal}
\end{equation}
Here $g_1,g_2$ encode the shared scalar constraint $t\in[-1,1]$. The generators $g_3,g_4$ encode $z_1\in[-1+\tfrac12 t,1+\tfrac12 t]$, while $g_7,g_8$ encode $z_2\in[-1-\tfrac12 t,1-\tfrac12 t]$. Finally, $g_5,g_6$ and $g_9,g_{10}$ correspond to the linear ReLU inequalities for $x_1=\relu(z_1)$ and $x_2=\relu(z_2)$, respectively.

The desired membership is
\begin{equation}
\widetilde p_e
=
\sigma_0+
\sum_{r=1}^{10}\sigma_r g_r
+
\tau_1 h_1+\tau_2 h_2,
\label{eq:app-reduced-cert-form}
\end{equation}
where
\begin{equation}
\sigma_0\ \text{is SOS with}\ \deg\sigma_0\le 4,
\qquad
\sigma_r\ \text{is SOS with}\ \deg\sigma_r\le 2\ (r=1,\dots,10),
\qquad
\deg\tau_1,\deg\tau_2\le 2.
\label{eq:app-degree-conditions}
\end{equation}
Equivalently, $\sigma_0$ admits a Gram representation on a basis of polynomials of degree at most $2$, while each $\sigma_r$ admits a Gram representation on a basis of polynomials of degree at most $1$. This is the concrete meaning of ``relaxation order $2$'' in the present example.

Solving this SDP with MOSEK yields the termination status \texttt{OPTIMAL}, objective value $0$, and a maximum coefficient residual in the polynomial identity equal to
\begin{equation}
1.49\times 10^{-8}.
\label{eq:app-max-residual}
\end{equation}
Hence \eqref{eq:app-reduced-cert-form} is satisfied numerically to high precision. This confirms that a local certificate of relaxation order $2$ can be constructed for the edge block in this instance.

For the isolated vertex $3$, one has
\begin{equation}
\Theta_3^{\mathrm{lin}}-m_3=x_3,
\label{eq:app-x3-generator}
\end{equation}
and since $x_3$ itself is one of the generators, it is handled trivially at the same degree. Therefore, from the additive decomposition along the matching in \eqref{eq:app-pgamma-decomp} and \eqref{eq:app-gammastar},
\begin{equation}
p_{\gamma^*}=\bigl(\Theta_e^{\mathrm{lin}}-m_e\bigr)+\bigl(\Theta_3^{\mathrm{lin}}-m_3\bigr)
\label{eq:app-global-gap-sum}
\end{equation}
holds, and by adding the certificates of the edge block and the isolated-vertex block, we obtain
\begin{equation}
p_{\gamma^*}\in \mathcal C_{\mathrm{sp}}^{(2)}.
\label{eq:app-global-cert}
\end{equation}

Therefore, this example shows that
\begin{enumerate}[label=(\roman*),leftmargin=2.4em]
\item the input-sharing graph is a matching and its unique edge is a regular rank-one edge,
\item the component decomposition along the matching yields $\gamma^*=13/2$, and
\item the same optimal value is reproduced by the order-$2$ block-sparse SOS relaxation.
\end{enumerate}
This example therefore provides a concrete illustration of the argument developed in Sections~3--5.